\documentclass[10pt]{amsart}
\usepackage[letterpaper,top=2cm,bottom=2cm,left=3cm,right=3cm,marginparwidth=1.75cm]{geometry}
\usepackage{amsmath}
\usepackage{amsfonts}
\usepackage{amsthm}
\usepackage{amssymb}
\usepackage{color}

\newtheorem{theorem}{Theorem}[section]

\newtheorem{lemma}[theorem]{Lemma}
\newtheorem{remark}[theorem]{Remark}
\newtheorem{corollary}[theorem]{Corollary}

\numberwithin{equation}{section}

\begin{document}
\title[Structure of the attractor of reaction-diffusion equations]{On $L^{\infty }$-estimates and the structure of the global attractor
for weak solutions of reaction-diffusion equations}
\author{Rub\'{e}n Caballero$^{1}$, Piotr Kalita$^{2}$ and Jos\'{e} Valero$%
^{1}$}
\address{$^{1}${\small Centro de Investigaci\'{o}n Operativa, Universidad
Miguel Hern\'{a}ndez de Elche,}\\
{\small Avda. Universidad s/n, 03202, Elche (Alicante), Spain}\\
{\small E.mails:\ ruben.caballero@umh.es, jvalero@umh.es}\\
$^{2}${\small Faculty of Mathematics and Computer Science, Jagiellonian
University, ul.}\\
{\small \L ojasiewicza 6, 30-348 Krakow, Poland}\\
{\small E.mail: piotr.kalita@ii.uj.edu.pl}}
\maketitle

\begin{abstract}
In this paper, we study the structure of the global attractor for weak and
regular solutions of a problem governed by a scalar semilinear reaction-diffusion equation with a  non-regular
nonlinearity, such that uniquness of solutions can fail to happen. First, using the Moser--Alikakos iterations we obtain
the estimates of the weak solutions in the space $L^{\infty }(\Omega)$. After that, using these estimates we
improve the existing results on the structure of the attractor. Finally,
estimates of the Hausdorff and fractal dimension of the attractor are
obtained.
\end{abstract}

\bigskip



\section{Introduction}

The study of the structure of the global attractor for semigroups or
semiflows generated by partial differential equations is one of the most
important and challenging problems for infinite-dimensional dissipative dynamical
systems.

In particular, the scalar reaction-diffusion equation%
\begin{equation}
u_{t}-\Delta u+f(u)=g  \label{Eq0}
\end{equation}%
with Dirichlet boundary conditions in a bounded domain $\Omega \subset 
\mathbb{R}^{d}$ has been widely studied and lots of deep results were
obtained so far by many authors concerning this task. When the function $f$
is $C^{1}$ and some extra assumptions are imposed on the derivative, this
problem generates a semigroup of operators and its global attractor can be
characterized by the unstable set of the set of stationary points (see e.g. 
\cite{BabinVishik85, BabinVishik89, Temam}). In more
particular situations, the attractor has been fully described in terms of
the heteroclinic connections of a finite set of stationary points (see, for
example, \cite{Fiedler85, Fiedler96, Henry85, Rocha88,  Rocha91} among others).

When the function $f$ is less regular, uniqueness can fail to happen or can
be very hard to check. In both cases, we must define a multivalued semiflow
instead of a semigroup, and studying the structure of the global attractor
is much more difficult. If the function $f$ is discontinuous and of the
Heaviside type, several interesting results were proved in \cite{ArrRBVal}
(see also \cite{Val21}\ for the nonautonomous case). In this paper, we are
interested in the situation where $f$ is a continuous function satisfying
the standard growth and dissipativity conditions (\ref{Growth})-(\ref{Diss}).
As neither monotonicity assumptions on $f$ nor growth conditions on $f^{\prime }$
are made, it is not possible to prove uniqueness of solutions of the Cauchy
problem. 

For the similar setup, in \cite{KKV14}, the problem for the bounded domain $\Omega\subset \mathbb{R}^3$ was considered. There, under an additional restriction on the constant $p$
in (\ref{Growth})-(\ref{Diss}) which determines the growth of $f$, and for $g$ in the space $%
L^{2}(\Omega)$, it was established that the global
attractor for the semiflow generated by the class of regular solutions consists of the unstable set of the set of stationary points,
extending in this way the previous results known in the single-valued situation.
In \cite[Theorem 6]{KKV15}, this characterization was given for weak
solutions either assuming an additional restriction on $p$ or taking 
$g$ in $L^{\infty }(\Omega)$. Our intention in the present paper is to obtain such result
in arbitrary space dimension of the domain $\Omega$ and without any additional restriction on the constant $p$ (that is, $p$ just
satifies $p\geq 2$). Thus we fill in the gap between the weak and regular solutions by showing that the two classes coincide. The restriction that we need when the dimension $d$ of the domain is
greater than $3$, is that the forcing term $g$ has to be in $L^{s}(\Omega )\ \ $with$\ \
s>d/2$, so for $g\in L^{2}$ the problem remains open when $d\geq 4$. Still, the result generalizes and encompasses the corresponding ones in \cite{KKV14, KKV15}. 

The crucial ingredient that we use to solve this problem relies in obtaining the suitable estimates of weak
solutions in the space $L^{\infty }(\Omega)$, a result which is also new and quite
interesting by itself. These bounds are obtained using the technique of Moser--Alikakos iterations. The spirit of the methodology, which consists in obtaining the uniform in $m$ and bounds in $L^m(\Omega)$ by gradual increasing of $m$ up to infinity is due to J. Moser \cite{Moser}. While the work of J. Moser concerned the elliptic problems, the results for the parabolic problems, in the same vein, were obtained by Alikakos \cite{Ali1, Ali2}.  Modern rendition of the Alikakos method, which is also a direct inspiration for the present work, can be found in the monograph \cite[Chapter 16]{Quittner} and in the article \cite{Rothe} which are the direct inspitation for the estimates obtained in `this work.

The novelty of the uniform in time $L^\infty(\Omega)$ estimate obtained in this paper, in Theorem \ref{LinfEst}, with respect to \cite{Quittner, Rothe} is twofold:
\begin{itemize} \item In \cite{Quittner, Rothe} the authors assume the $L^\infty(\Omega)$ regularity of the initial data to obtain the uniform in time $L^{\infty}(\Omega)$ bounds for positive $t$. Here we assume only $L^2(\Omega)$ regularity of the initial data and we carefully derive the Alikakos-type estimate in $L^{2A}(\Omega)$ for some $A>1$ after a small time $t$. Iteratively, if the initial data belongs to $L^m(\Omega)$ for some $m\geq 2$, we get the bound in  $L^{Am}(\Omega)$ again, after some time which depends on $m$ (see Lemma \ref{lem:33}). The novelty consists in the careful calculation of the lengths of time intervals needed for the regularity gain in each step of the iteration  such that the total length of the infinite number of steps in the bootstrap procedure required for the gain of regularity up to $L^\infty(\Omega)$ is still  arbitrarily small.
	\item  The key concept of the Alikakos method is to test the equation by $u^r$ for some appropriately chosen $r$. If the weak solution is unique, then, typically, it is possible to take such test function in the weak form of the problem. In the case of nonuniqueness, however, we know nothing about the regularity of all weak solutions and we are not  allowed to test by functions which do not have the regularity required in the definition of the weak solution. To overcome this difficulty we test the weak form of the problem by the power of the  truncation of the solution $(T_ku)^r$, and we derive the estimates independent on the truncation level. The novelty consists in showing the correctness of such procedure. Thus, uniform in time $L^\infty(\Omega)$ bounds hold for all, possibly nonunique, weak solutions of the problem without the need to impose any additional assertions in the definition of the weak solution (for example they do not need to be limits of some approximation procedures).  
\end{itemize}
It is worth mentioning that these estimates imply
that the global attractor attractor is bounded in $L^{\infty }$: the method allows us to obtain the explicit estimates of the radius of the $L^\infty(\Omega)$ ball centered at zero which contains the global attractor.  

The $L^\infty(\Omega)$ bounds allow us to deduce that if we assume
additionally that $f$ is locally one-sided Lipschitz, then the uniqueness of weak
solutions is true for any initial condition inside the global attractor. For all solutions nonuniqueness can only possibly occur only near the initial moment of time for the initial data belonging to $L^2(\Omega)\setminus L^\infty(\Omega)$. 

The uniqueness of the solution on the global attractor allows us to deduce, inder local (two-sided) Lipschitz condition on the nonlinearity that the Hausdorff
and fractal dimensions of the global attractor are finite. In this way, we
improve the known results given in \cite{Babin83, BabinVishik85, Marion}, in
which the function $f$ was differentiable with an extra assumption on the
derivative (see \cite{Robinson} as well), and also in \cite{BalVal}, in
which $f$ was globally Lipschitz. 

This paper is organized as follows. In Section 2, we recall the previous results for equation (\ref{Eq0}) on
global attractors and their structure in the case of the solution non-uniqueness. In
Section 3, we obtain uniform estimates of weak solutions in the space $%
L^{\infty }$. In Section 4, we improve the previous results on the structure
of the global attractor for weak and regular solutions. In Section 5, we
obtain an estimate of the Hausdorff and fractal dimensions of the attractor.

\section{Semilinear heat equation: problem setup and previous results in 
\protect\cite{KKV14, KKV15, KV06, KV09}.}

We denote $\mathbb{R}_+ = [0,\infty)$. For the real numbers $p, s, q \in
(1,\infty)$ we will denote by prime the conjugate exponents, i.e., for
example $1/p+1/p^{\prime }=1$. Let $p\geq 2$ and let $\Omega\subset \mathbb{R%
}^d$ be a bounded domain with sufficiently smooth (Lipschitz) boundary. We
assume that $d\geq 1$. We consider the problem 
\begin{equation}
\left\{ 
\begin{array}{l}
u_{t}-\Delta u+f(u)=g,\quad x\in \Omega ,\ t>0, \\ 
u|_{\partial \Omega }=0, \\ 
u(0)=u_{0},%
\end{array}
\right.  \label{Eq}
\end{equation}

The initial data $u_{0}$ is assumed to belong to $L^{2}(\Omega )$.
Assumptions on $f$ are the following 
\begin{equation}
f\in C(\mathbb{R}),  \label{Contf}
\end{equation}%
\begin{equation}
|f(u)|\leq C_{1}(1+|u|^{p-1})\qquad \text{for every}\qquad u\in \mathbb{R},
\label{Growth}
\end{equation}%
\begin{equation}
f(u)u\geq \alpha \left\vert u\right\vert ^{p}-C_{2}\qquad \text{for every}%
\qquad u\in \mathbb{R},  \label{Diss}
\end{equation}%
with $C_{1},C_{2},\alpha >0,\ p\geq 2$. 
As for the forcing term $g$ we assume that 
\begin{equation}
g\in L^{s}(\Omega )\ \ \text{for}\ \ s>\frac{d}{2}\ \ \text{if}\ \ d\geq 2\
\ \text{and}\ \ g\in L^{1}(\Omega )\ \ \text{if}\ \ d=1.\   \label{g}
\end{equation}

We show that assuming only \eqref{Contf}--\eqref{Diss}, under the regularity of $g,$ given by \eqref{g} the solution instantaneously starts to belongs
to $L^{\infty }(\Omega)$ with respect to the space variable. This implies that the global attractor is a bounded set in $L^\infty(\Omega)$. 
Although the global attractor becomes more regular then merely 
$L^{2}(\Omega )$, the dynamics on it can be, possibly, multivalued. Usually
either monotonicity of $f$ or growth condition of the derivative $f^{\prime
} $ is needed for the uniqueness of weak solutions.  Obtaining $L^\infty$ bounds implies that it is
enough to assume that $f$ is locally lipschitz to get to deduce that:

\begin{itemize}
\item The dynamics on the global attractor is single valued.

\item Nonuniqueness can occur only at initial time $t=0$, no new trajectory
can branch out from the trajectory of the problem at positive time.

\item The fractal and Hausdorff dimensions of the global attractor are
finite.
\end{itemize}

\subsection{Attractor for weak solutions:\ existence}

In \cite{KV06} the existence and connectedness of the global attractor for
the multivalued semiflow generated by weak solutions was established.

A function $u\in L_{loc}^{2}(0,+\infty ;H_{0}^{1}(\Omega ))\bigcap
L_{loc}^{p}(0,+\infty ;L^{p}(\Omega ))$ is called a weak solution of (\ref%
{Eq}) on $[0,+\infty )$ if for all $T>0\,,\ v\in H_{0}^{1}\left( \Omega
\right) \cap L^{p}(\Omega ),\,\eta \in C_{0}^{\infty }(0,T),$ 
\begin{equation}
-\int\limits_{0}^{T}(u,v)\eta _{t}dt+\int\limits_{0}^{T}\left(
(u,v)_{H_{0}^{1}(\Omega )}+(f(u),v)-(g,v)\right) \eta dt=0.  \label{EqSol}
\end{equation}
Let $K^{+}$ be the set of all weak solutions. Define the map $G:\mathbb{R}
^{+}\times L^{2}(\Omega )\rightarrow P(L^{2}(\Omega ))$ as 
\begin{equation*}
G\left( t,u_{0}\right) =\left\{ u\left( t\right) :u\left(\text{\textperiodcentered }\right) \in K^{+}\text{ such that }u\left( 0\right)
=u_{0}\right\},
\end{equation*}
where $P\left( X\right) $ stands for the set of all non-empty subsets of the
space $X.$

We recall that a map $\gamma :\mathbb{R}\rightarrow L^{2}(\Omega )$ is
called a complete trajectory of $K^{+}$ if 
\begin{equation*}
\,\,\,\ \gamma (\cdot +h)|_{[0,+\infty )}\in K^{+}\ \ \text{for every}\ \
h\in \mathbb{R},
\end{equation*}%
that is, if $\gamma |_{[\tau ,+\infty )}$ is a weak solution of (\ref{Eq})
on $[\tau ,+\infty ),\,\ \forall \tau \in \mathbb{R}$. We denote \ by $%
\mathbb{F}$ the set of all complete trajectories of $K^{+}$, and let $%
\mathbb{K}$ be the set of all bounded (in the $L^{2}(\Omega )$ norm)
complete trajectories.

The following results are proved in \cite{KKV14, KV06, KV09}:

\begin{itemize}
\item The operator $G$ is a strict multivalued semiflow \cite[Lemma 9]{KV06}.

\item $G$ possesses a global compact invariant attractor $\mathcal{A}$ \cite[
Theorem 10]{KV06}.

\item The global attractor $\mathcal{A}$ is connected \cite[Theorem 28]{KV09}.

\item $\mathcal{A}=\left\{ \gamma (0)\,:\,\gamma (\cdot )\in \mathbb{K}
\right\} =\bigcup\limits_{t\in \mathbb{R}}\left\{ \gamma (t)\,:\,\gamma
(\cdot )\in \mathbb{K}\right\} \ $\cite[Theorem 4]{KKV14}.
\end{itemize}

Note that the results of \cite{KKV14, KV06, KV09} require that $g\in
L^{2}(\Omega )$. We assume the condition \eqref{g} which is weaker for $d\in
\{1,2,3\}$, namely $g\in L^{1}(\Omega )$ if $d=1$ and $g\in L^{s}(\Omega )$
for $s>1$ if $d=2$ and $s>\frac{3}{2}$ for $d=3$. However, due to the fact
that such $g$ always belongs to the dual of the Lebesgue space in which $%
H_{0}^{1}(\Omega )$ is embedded, results of \cite{KKV14, KV06, KV09} stay
valid under assumption \eqref{g} for $d\in \{1,2,3\}$. If $d\geq 4$, the
condition \eqref{g} gives us an additional restriction with respect to the
assumption $g\in L^{2}(\Omega )$ used in the proof of the existence of the
attractor for the weak solutions because it is possible that $g\in
L^{2}(\Omega )$ and $g\notin L^{2}(\Omega )$ for some $s>\frac{d}{2}$. The
results of this article are obtained under this restriction.

\subsection{Regular solutions:\ existence and structure of the attractor}

In \cite{KKV14} the existence and structure of the global attractor for
regular solutions was established.

The function $u\in L_{loc}^{2}(0,+\infty ;H_{0}^{1}(\Omega ))\bigcap
L_{loc}^{p}(0,+\infty ;L^{p}(\Omega ))$ is called a regular solution of (\ref%
{Eq}) on $[0,+\infty )$ if for all $T>0,\,v\in H_{0}^{1}(\Omega )\cap
L^{p}(\Omega )\,\ $and $\eta \in C_{0}^{\infty }(0,T)$ we have 
\begin{equation}
-\int\limits_{0}^{T}(u,v)\eta _{t}dt+\int\limits_{0}^{T}\left(
(u,v)_{H_{0}^{1}(\Omega )}+(f(u),v)-(g,v)\right) \eta dt=0,  \label{EqSol2}
\end{equation}%
and 
\begin{equation*}
u\in L^{\infty }\left( \varepsilon ,T;H_{0}^{1}\left( \Omega \right) \right)
\ \ \text{and}\ \ u_{t}\in L^{2}\left( \varepsilon ,T;L^{2}\left( \Omega
\right) \right) \ \ \text{for every}\ \ 0<\varepsilon <T.
\end{equation*}

It is assumed in \cite{KKV14} that $g\in L^{2}(\Omega )$ and that the
constant $p$ satisfies the following assumption: 
\begin{equation}
p\leq \frac{2d-2}{d-2}\text{ if }d\geq 3.  \label{Condp}
\end{equation}

For any $u\in L^{2}\left( \Omega \right) $ there exists at least one regular
solution $u\left( \text{\textperiodcentered }\right) $ such that $u\left(
0\right) =u_{0}$ \cite[Theorem 22]{KKV14}. Let $K_{r}^{+}$ be the set of all
regular solutions. Define the map $G_{r}:\mathbb{R}^{+}\times L^{2}(\Omega
)\rightarrow P(L^{2}(\Omega ))$ as 
\begin{equation*}
G_{r}\left( t,u_{0}\right) =\left\{ u\left( t\right) :u\left(\text{\textperiodcentered }\right) \in K_{r}^{+}\text{ such that }
u(0)=u_{0}\right\} .
\end{equation*}

The map $\gamma :\mathbb{R}\rightarrow L^{2}\left( \Omega \right) $ is
called a complete trajectory of $K_{r}^{+}$ if $\gamma \left( \text{
\textperiodcentered }+h\right) \mid _{\lbrack 0,+\infty )}\in K_{r}^{+}$ for
any $h\in \mathbb{R}.$ Denote by $\mathbb{F}_{r}$ the set of all complete
trajectories of $K_{r}^{+},$ while $\mathbb{K}_{r}$ will be the set of all
bounded complete trajectories of $K_{r}^{+}.$ Also, let $\mathbb{K}_{r}^{1}$
be the set of all complete trajectories which are bounded in $%
H_{0}^{1}\left( \Omega \right) .$ It is known that $\mathbb{K}_{r}^{1}= 
\mathbb{K}_{r}$ \cite[Lemma 31]{KKV14}.

The following results are proved in \cite{KKV14}:

\begin{itemize}
\item The operator $G_{r}$ is a (posssibly non-strict) semiflow.

\item $G_{r}$ posseses a global a global attractor $\mathcal{A}_{r}$ \cite[
Theorem 29]{KKV14}. Moreover, $\mathcal{A}_{r}$ is compact in $%
H_{0}^{1}\left( \Omega \right) $ and 
\begin{equation*}
\text{dist}_{H_{0}^{1}(\Omega )}(G_{r}(t,B),\mathcal{A}_{r})\rightarrow 0%
\text{ as } t\rightarrow +\infty ,
\end{equation*}
for any bounded set $B.$

\item $\mathcal{A}_{r}=\{\gamma \left( 0\right) :\gamma \left( \text{
\textperiodcentered }\right) \in \mathbb{K}_{r}\}=\{\gamma \left( 0\right)
:\gamma \left( \text{\textperiodcentered }\right) \in \mathbb{K}
_{r}^{1}\}=\cup _{t\in \mathbb{R}}\{\gamma \left( t\right) :\gamma \left( 
\text{\textperiodcentered }\right) \in \mathbb{K}_{r}\}=\cup _{t\in \mathbb{%
R }}\{\gamma \left( t\right) :\gamma \left( \text{\textperiodcentered }%
\right) \in \mathbb{K}_{r}^{1}\}.$

\item $\mathcal{A}_{r}$ has the following structure: 
\begin{equation}
\mathcal{A}_{r}=M_{r}^{+}(\mathfrak{R})=M_{r}^{-}(\mathfrak{R}),  \label{Str}
\end{equation}
where $\mathfrak{R}$ is the set of equilibria of problem (\ref{Eq}) and 
\begin{equation*}
\begin{array}{c}
M_{r}^{-}(\mathfrak{R})=\left\{ z\,:\,\exists \gamma (\cdot )\in \mathbb{K}
_{r},\,\ \gamma (0)=z,\,\,\,\ \mathrm{dist}_{L^{2}(\Omega )}(\gamma (t), 
\mathfrak{R})\rightarrow 0,\,\ t\rightarrow +\infty \right\} , \\ 
M_{r}^{+}(\mathfrak{R})=\left\{ z\,:\,\exists \gamma (\cdot )\in \mathbb{F}
_{r},\,\ \gamma (0)=z,\,\,\,\ \mathrm{dist}_{L^{2}(\Omega )}(\gamma (t), 
\mathfrak{R})\rightarrow 0,\,\ t\rightarrow -\infty \right\} .%
\end{array}%
\end{equation*}
\end{itemize}

Using the $L^{\infty }$-estimates, we are going to prove these results
without the restriction \eqref{Condp}. The only restriction that we impose
is a condition on the regularity of the forcing term $g,$ which is a bit
stronger than assumption (\ref{g}) and is given by 
\begin{equation}
g\in L^{s}(\Omega )\ \ \text{for}\ \ s>\frac{d}{2}\ \ \text{if}\ \ d\geq 4\
\ \text{and}\ \ g\in L^{2}(\Omega )\ \ \text{if}\ \ 1\leq d\leq 3.\ 
\label{g2}
\end{equation}
This assumption makes it
possible to bootstrap the regularity further from $L^\infty(\Omega)$ to $H_{0}^{1}(\Omega )$ and
hence the weak solutions coincide with the class of regular ones.

\subsection{Weak solutions:\ structure of the attractor}

In \cite{KKV15} the structure of the global attractor for weak solutions was
established under additional assumptions by proving that every weak solution
is a regular solution. Although in \cite{KKV15} only the three-dimensional
case is considered, the result can be given for a general $d$. So, let $d=3$ and assume that 
\begin{equation}
2\leq p\leq 3.  \label{Condp2}
\end{equation}
Then the following facts are proved:

\begin{itemize}
\item Any weak solution $u\left( \text{\textperiodcentered }\right) $ is a
regular solution \cite[Lemma 3]{KKV15}, so $G=G_{r}$, and then $\mathcal{A}= 
\mathcal{A}_{r}$ $.$

\item Equality (\ref{Str}) holds true for $\mathcal{A}$ \cite[Theorem 6]%
{KKV15}.
\end{itemize}

\bigskip

The structure of $\mathcal{A}$ was also obtained in \cite[Theorem 5]{KKV15}
by assuming that $d=3$ and $g\in L^{\infty }(\Omega ).$ Using the $L^{\infty }$-estimates, we are going to prove these results
without any restrictions on $p\geq 2$ and under assumption \eqref{g2}. So,
we prove that the gap between weak solutions and their attractors and
regular solutions and their attractors may only hold for $d\geq 4$ in the
case when $g\in L^{2}(\Omega )$ but $g\notin L^{s}(\Omega )$ for some $s>%
\frac{d}{2}$.

\section{Moser--Alikakos iterations and $L^\infty$ estimates for weak
solutions}

In this section we prove the $L^{\infty }(\Omega)$ estimate on $u(t)$ after
arbitrarily small time $\tau $ from the initial time. We assume conditions %
\eqref{Contf}--\eqref{g}. Under the above assumptions we will obtain the
following result.

\begin{theorem}
\label{LinfEst}If $u$ is a weak solution of (\ref{Eq}) with $\sup_{s\geq
t_{1}}\Vert u(s)\Vert _{2}<\infty $, then for every $\tau >0$ there exists a
constant $D(\tau )$ such that for every $t\geq t_{1}+\tau $ we have $u(t)\in
L^{\infty }(\Omega)$ and 
\begin{equation*}
\Vert u(t)\Vert _{\infty }\leq D(\tau )\left( \sup_{s\geq t_{1}}\Vert
u(s)\Vert _{2}+1\right) .
\end{equation*}
\end{theorem}

The main tool to get the above result will be the iterative inequality given
in the following lemma.

\begin{lemma}
\label{lem:33} Let $A=2$ if $d\in \{1,2,3\}$ and $A=\frac{d}{d-2}$ if $d\geq
4$. Assume that $m\geq 2$, $t_1 \geq 0$, $\tau > 0$ and that a weak solution
of \eqref{Eq} satisfies $\sup_{t\in [t_1,t_1+\tau]}\|u^+(t)\|_m < \infty$.
Then $u^+(t_1+\tau) \in L^{Am}(\Omega)$ and there exists $D>0$ such that for 
$\delta > 0$ satisfying $\tau \geq \frac{\delta}{m^D}$ there is $D(\delta) >
0$ for which the following estimate holds 
\begin{equation*}
\max\left\{\|u^+(t_1+\tau)\|_{Am},D\right\} \leq D(\delta)^{\frac{1}{m}}m^%
\frac{D}{m}\max\left\{\sup_{t\in [t_1, t_1+\tau]}\|u^+(t)\|_{m},D\right\}.
\end{equation*}
\end{lemma}

The technique to get the above result was initially developed by Alikakos 
\cite{Ali1, Ali2} as a variant of the Moser iteration scheme, see also \cite%
{Rothe} and \cite[Theorem 16.4]{Quittner}. While in aforementioned papers
the authors prove that $L^{\infty }(\Omega)$ bound on initial data is \textit{%
preserved} in time, we obtain the result that for initial data in $L^{2}(\Omega)$
taken at $t=0$, the $L^{\infty }$ regularity and the corresponding bound
holds for positive times. We will use the truncation operator at level $k$
denoted as $T_{k}:\mathbb{R}\rightarrow \lbrack -k,k]$ and defined as 
\begin{equation*}
T_{k}u=%
\begin{cases}
u\ \ \text{if}\ \ |u|\leq k \\ 
k\frac{u}{|u|}\ \ \text{otherwise},%
\end{cases}%
\end{equation*}%
as well as the functions $u^{+}=\max \{u,0\}$ and $u^{-}=-\min \{u,0\}.$
Note that if $u\in W^{1,q}(\Omega )$ for $q\geq 1$, then, cf., e.g., \cite[%
Proposition A.2.4]{Ambrosetti}, 
\begin{equation*}
\nabla T_{k}u=%
\begin{cases}
\nabla u\ \ \text{for a.e.}\ x\ \text{with}\ \ |u(x)|<k, \\ 
0\ \ \text{for a.e.}\ x\ \text{with}\ \ |u(x)|\geq k.%
\end{cases}%
\end{equation*}%
Using the same result we deduce that, if $m\geq 1$, then $u\in
W^{1,p}(\Omega )$ implies that 
\begin{equation*}
\nabla ((T_{k}u^{+})^{2m-1})=(2m-1)(T_{k}u^{+})^{2m-2}\nabla (T_{k}u^{+})\in
(L^{p}(\Omega ))^{d}.
\end{equation*}%
This makes it possible to test the weak form of the equation with $%
(T_{k}u^{+})^{2m-1}$.

We begin from the integration by parts formula valid for $m \geq 1$ which is
a nonlinear generalization of the formula (25) from \cite{KKV14}.

\begin{lemma}
\label{lem:time} Suppose that $\tau <t$ and $u\in L^{2}(\tau
,t;H_{0}^{1}(\Omega ))\cap L^{p}(\tau ,t;L^{p}(\Omega ))$ is such that the
distributional time derivative has regularity $u_{t}\in L^{2}(\tau
;t;H^{-1}(\Omega ))+L^{p^{\prime }}(\tau ,t;L^{p^{\prime }}(\Omega ))$. Then 
\begin{equation*}
\int_{\tau }^{t}\langle u_{t},(T_{k}u^{+})^{2m-1}\rangle \eta
\,ds=-\int_{\tau }^{t}\int_{\Omega }\varphi _{k,m}(u^{+}(s,x))\,dx\,\eta
^{\prime }\,ds,
\end{equation*}%
for every $\eta \in C_{0}^{\infty }(\tau ,T)$ and for every $k>0$, where,
for $s\geq 0$ 
\begin{equation*}
\varphi _{k,m}(s)=\int_{0}^{s}T_{k}(r)^{2m-1}\,dr=%
\begin{cases}
\frac{s^{2m}}{2m}\ \text{if}\ s\leq k, \\ 
\frac{k^{2m}}{2m}+(s-k)k^{2m-1}\ \text{if}\ s>k.%
\end{cases}%
\end{equation*}
\end{lemma}

\begin{proof}
By a mollification argument we can construct a sequence $u^{n}\in C^{\infty
}([\tau ,t];H_{0}^{1}(\Omega )\cap L^{p}(\Omega ))$ such that 
\begin{align*}
u^{n}& \rightarrow u\ \ \text{in}\ \ L^{2}(\tau ,t;H_{0}^{1}(\Omega ))\cap
L^{p}(\tau ,t;L^{p}(\Omega )), \\
u_{t}^{n}& \rightarrow u_{t}\ \ \text{in}\ \ L^{2}(\tau ;t;H^{-1}(\Omega
))+L^{p^{\prime }}(\tau ,t;L^{p^{\prime }}(\Omega )).
\end{align*}%
For a subsequence still denoted by $n$ and for a.e. $s\in (\tau ,t)$, we
have $u^{n}(s)\rightarrow u(s)$ in $H_{0}^{1}(\Omega )\cap L^{p}(\Omega )$, $%
u_{t}^{n}(s)\rightarrow u_{t}(s)$ in $H^{-1}(\Omega )+L^{p^{\prime }}(\Omega
)$ with $\Vert u^{n}(s)\Vert _{H_{0}^{1}}\leq g_{1}(s)$ and $\Vert
u^{n}(s)\Vert _{p}\leq g_{2}(s)$, where $g_{1}\in L^{2}(\tau ,t)$ and $%
g_{2}\in L^{p}(\tau ,t)$. Moreover $u_{t}^{n}=a^{n}+b^{n}$, where $a^{n}\in
L^{2}(\tau ,t;H^{-1}(\Omega ))$ and $b^{n}\in L^{p^{\prime }}(\tau
,t;L^{p^{\prime }(\Omega )})$, with $\Vert a^{n}(s)\Vert _{H^{-1}}\leq
g_{3}(s)$ and $\Vert b^{n}(s)\Vert _{L^{p^{\prime }}}\leq g_{4}(s)$ with $%
g_{3}\in L^{2}(\tau ,t)$ and $g_{4}\in L^{p^{\prime }}(\tau ,t)$. The
required integration by parts formula holds for $u^{n}$, i.e. 
\begin{equation}
\int_{\tau }^{t}\langle u_{t}^{n}(s),(T_{k}(u^{n}(s))^{+})^{2m-1}\rangle
\eta \,ds=-\int_{\tau }^{t}\int_{\Omega }\varphi
_{k,m}((u^{n}(s,x))^{+})\,dx\,\eta ^{\prime }\,ds,  \label{partsn}
\end{equation}%
We will prove that we can use the Lebesgue dominated convergence theorem to
pass to the limit and get the assertion for $u$. We first prove that $%
(T_{k}(u^{n}(s))^{+})^{2m-1}\rightarrow (T_{k}(u(s))^{+})^{2m-1}$ in $%
H_{0}^{1}(\Omega )\cap L^{p}(\Omega )$. Indeed 
\begin{equation*}
|(T_{k}(u^{n}(s))^{+})^{2m-1}-(T_{k}(u(s))^{+})^{2m-1}|^{p}\leq
(2k^{2m-1})^{p},
\end{equation*}%
the pointwise convergence for a.e. $x$, for a subsequence, follows from the
fact that $u^{n}(s)\rightarrow u(s)$ in $L^{p}(\Omega )$, and the a.e. convergence for the whole sequence follows from the uniqueness of
the limit. Moreover 
\begin{align*}
& \int_{\Omega }|\nabla (T_{k}(u^{n}(s))^{+})^{2m-1}-\nabla
(T_{k}(u(s))^{+})^{2m-1}|^{2}\,dx \\
& \ \ \leq (2m-1)^{2}\int_{\Omega }|(T_{k}(u^{n}(s))^{+})^{2m-2}\nabla
(T_{k}(u^{n}(s))^{+})-(T_{k}(u(s))^{+})^{2m-2}\nabla
(T_{k}(u(s))^{+})|^{2}\,dx \\
& \ \ \leq 2(2m-1)^{2}\left( \int_{\Omega
}(T_{k}(u^{n}(s))^{+})^{4m-4}|\nabla (T_{k}(u^{n}(s))^{+})-\nabla
(T_{k}(u(s))^{+})|^{2}\,dx\right. \\
& \qquad \qquad \left. +\int_{\Omega
}((T_{k}(u^{n}(s))^{+})^{2m-2}-(T_{k}(u(s))^{+})^{2m-2})^{2}|\nabla
(T_{k}(u(s))^{+})|^{2}\,dx\right) ,
\end{align*}%
and the Lebesque dominated convergence theorem implies that the last
expression tends to zero as $n\rightarrow \infty $. We deduce that we can
pass to the limit for a.e. $s$ in the left-hand side in \eqref{partsn}.
Moreover, 
\begin{align*}
& \langle u_{t}^{n}(s),(T_{k}(u^{n}(s))^{+})^{2m-1}\rangle
_{(H^{-1}+L^{p^{\prime }})\times (H_{0}^{1}\cap L^{p})} \\
& \ \ =\langle a^{n}(s),(T_{k}(u^{n}(s))^{+})^{2m-1}\rangle _{H^{-1}\times
H_{0}^{1}}+\langle b^{n}(s),(T_{k}(u^{n}(s))^{+})^{2m-1}\rangle
_{L^{p^{\prime }}\times L^{p}} \\
& \ \ \leq \Vert a^{n}(s)\Vert _{H^{-1}}\Vert
((T_{k}u^{n}(s))^{+})^{2m-1}\Vert _{H_{0}^{1}}+\Vert b^{n}(s)\Vert
_{p^{\prime }}\Vert ((T_{k}u^{n}(s))^{+})^{2m-1}\Vert _{p} \\
& \ \ \leq g_{3}(s)(2m-1)k^{2m-2}\Vert \nabla ((T_{k}u^{n}(s))^{+})\Vert
_{2}+g_{4}(s)k^{2m-1}|\Omega |^{\frac{1}{p}} \\
& \ \ \leq g_{3}(s)(2m-1)k^{2m-2}\Vert \nabla u^{n}(s)\Vert
_{2}+g_{4}(s)k^{2m-1}|\Omega |^{\frac{1}{p}} \\
& \ \ \leq (2m-1)k^{2m-2}g_{1}(s)g_{3}(s)+k^{2m-1}|\Omega |^{\frac{1}{p}%
}g_{4}(s)\in L^{1}(\tau ,t),
\end{align*}%
and we can use the Lebesgue dominated convergence theorem to pass to the
limit in the left-hand side of \eqref{partsn}. To pass to the limit on the
right-hand side note that 
\begin{eqnarray*}
\left\vert \int_{\Omega }\varphi _{k,m}((u^{n}(s,x))^{+})\,dx\right\vert
&\leq &\int_{\Omega }k^{2m-1}|u^{n}(s,x)|dx\leq k^{2m-1}\Vert u^{n}(s)\Vert
_{2}|\Omega |^{\frac{1}{2}} \\
&\leq &k^{2m-1}\Vert g_{1}(s)\Vert _{2}|\Omega |^{\frac{1}{2}}\in L^{1}(\tau
,t).
\end{eqnarray*}%
Moreover, for a.e. $s\in (\tau ,t)$ and for a.e. $x\in \Omega $, for a
subsequence (which may depend on $s$), $u^{n}(s,x)\rightarrow u(s,x)$ with $%
|u^{n}(s,x)|\leq h(x)\in L^{p}(\Omega )$ (where also $h$ depends on $s$).
This implies that 
\begin{equation*}
\varphi _{k,m}((u^{n}(s,x))^{+})\leq k^{2m-1}|u^{n}(s,x)|\leq
k^{2m-1}h(x)\in L^{1}(\Omega ),
\end{equation*}%
and $\varphi _{k,m}((u^{n}(s,x))^{+})\rightarrow \varphi
_{k,m}((u(s,x))^{+}) $ for a.e. $x$. We deduce that for a.e. $s\in (\tau ,t)$
we have 
\begin{equation*}
\int_{\Omega }\varphi _{k,m}((u^{n}(s,x))^{+})\,dx\rightarrow \int_{\Omega
}\varphi _{k,m}((u(s,x))^{+})\,dx
\end{equation*}%
and the convergence holds for the whole sequence from the uniqueness of the
limit. The proof is complete.
\end{proof}

\bigskip

Define the quantity 
\begin{equation*}
\Phi _{k,m}(u)=\int_{\Omega }\varphi _{k,m}(u(x))\,dx.
\end{equation*}%
By a straightforward calculation we obtain the following lemma

\begin{lemma}
\label{lem:Phi} If $m\geq 1$ then 
\begin{equation*}
\frac{1}{2m}\|T_ku^+\|_{2m}^{2m} \leq \Phi_{k,m}(u^+) \leq \frac{1}{2m}
\|u^+\|_{2m}^{2m}
\end{equation*}
\end{lemma}

Since for every $k>0$ and $m\geq 1$ we have the regularity 
\begin{equation*}
(T_{k}u^{+})^{2m-1}\in L_{loc}^{2}(\mathbb{R}_{+};H_{0}^{1}(\Omega ))\cap
L_{loc}^{\infty }(\mathbb{R}_{+};L^{\infty }(\Omega )),
\end{equation*}%
we can use $(T_{k}u^{+})^{2m-1}$ as the test function in the weak form of %
\eqref{Eq}. The use of the truncation in the test function is inspired by 
\cite{Blan, BlanMu}, but here we combine it with the Alikakos--Moser
iterative scheme. The letter $D$ in the following theorem will be a generic
constant which may change from line to line and is independent on the
initial data of the problem, and on the constants $k,m,\delta $. By $D(\delta )$
we will denote the constant depending on $\delta $ but not on $k,m$ and the
initial data.

We are in position to present the proof of Lemma \ref{lem:33}.

\begin{proof}
\textit{(of Lemma \ref{lem:33})} Testing the weak form of the equation by $%
(T_{k}u^{+})^{2m-1}$ we obtain 
\begin{equation*}
\langle u_{t},(T_{k}u^{+})^{2m-1}\rangle +(2m-1)(\nabla
u,(T_{k}u^{+})^{2m-2}\nabla
(T_{k}u^{+}))+(f(u),(T_{k}u^{+})^{2m-1})=(g,(T_{k}u^{+})^{2m-1}),
\end{equation*}%
for almost every $t>t_{0}$. We treat all terms separately. First, we have 
\begin{eqnarray*}
(2m-1)(\nabla u,(T_{k}u^{+})^{2m-2}\nabla (T_{k}u^{+}))
&=&(2m-1)\int_{\Omega }|\nabla (T_{k}u^{+})|^{2}|T_{k}u^{+}|^{2m-2}\,dx \\
&=&\frac{2m-1}{m^{2}}\int_{\Omega }|\nabla (T_{k}u^{+})^{m}|^{2}\,dx
\end{eqnarray*}%
Now 
\begin{align*}
& (f(u),(T_{k}u^{+})^{2m-1})=\int_{u\in
(0,k]}f(u)uu^{2m-2}\,dx+\int_{u>k}f(u)u^{2m-1}\frac{k^{2m-1}}{u^{2m-1}}\,dx
\\
& \ \geq \alpha \int_{u\in (0,k]}u^{2m-2+p}\,dx-C_{2}\int_{u\in
(0,k]}u^{2m-2}\,dx+\alpha \int_{u>k}u^{p-1}k^{2m-1}dx-C_{2}\int_{u>k}\frac{%
k^{2m-1}}{u}dx.
\end{align*}%
Using the inequality 
\begin{equation*}
a^{2m-2}\leq \epsilon ^{\frac{2m-2+p}{2m-2}}\frac{2m-2}{2m-2+p}a^{2m-2+p}+%
\frac{p}{\epsilon ^{\frac{2m-2+p}{p}}(2m-2+p)},
\end{equation*}%
where we choose $\epsilon $ such that the constant in front of $a^{2m-2+p}$
is equal to $\frac{\alpha }{2C_{2}}$, we obtain 
\begin{equation*}
a^{2m-2}\leq \frac{\alpha }{2C_{2}}a^{2m-2+p}+\frac{p(2C_{2})^{\frac{2m-2}{p}%
}(2m-2)^{\frac{2m-2}{p}}}{\alpha ^{\frac{2m-2}{p}}(2m-2+p)^{\frac{2m-2}{p}%
}(2m-2+p)}\leq \frac{\alpha }{2C_{2}}a^{2m-2+p}+\left( \frac{2C_{2}}{\alpha }%
\right) ^{\frac{2m-2}{p}}.
\end{equation*}%
Moreover, without loss of generality we can assume that $k\geq \left( \frac{%
2C_{2}}{\alpha }\right) ^{\frac{1}{p}}$, whence for $u>k$ it follows that $%
C_{2}\frac{k^{2m-1}}{u}\leq \frac{\alpha }{2}u^{p-1}k^{2m-1}$. We deduce 
\begin{equation*}
(f(u),(T_{k}u^{+})^{2m-1})\geq \frac{\alpha }{2}\int_{u\in
(0,k]}u^{2m-2+p}\,dx+\frac{\alpha }{2}\int_{u>k}u^{p-1}k^{2m-1}\,dx-|\Omega
|C_{2}\left( \frac{2C_{2}}{\alpha }\right) ^{\frac{2m-2}{p}}.
\end{equation*}%
Now note that 
\begin{align*}
& \text{if}\ \ u\in (0,k]\ \text{then}\ \ u^{2m-2+p}\geq \frac{u^{2m}}{2m}-1,
\\
& \text{if}\ \ u>k\ \text{then}\ u^{p-1}k^{2m-1}\geq \frac{k^{2m}}{2m}%
+(u-k)k^{2m-1}.
\end{align*}%
This means that 
\begin{equation*}
(f(u),(T_{k}u^{+})^{2m-1})\geq \frac{\alpha }{2}\Phi _{k,m}(u^{+})-\frac{%
\alpha }{2}|\Omega |-|\Omega |C_{2}\left( \frac{2C_{2}}{\alpha }\right) ^{%
\frac{2m-2}{p}}.
\end{equation*}%
We estimate the term with $g$. In the following part of the argument we
assume that $d\geq 3$. In the end of the proof we explain how to modify the
argument in order to account for the case of $d\in \{1,2\}$. We have: 
\begin{equation*}
(g,(T_{k}u^{+})^{2m-1})\leq \Vert g\Vert _{s}\Vert T_{k}u^{+}\Vert
_{s^{\prime }(2m-1)}^{2m-1}.
\end{equation*}%
The term with time derivative is estimated using Lemma \ref{lem:time}. We
obtain 
\begin{equation*}
\frac{d}{dt}\Phi _{k,m}(u^{+})+\frac{2m-1}{m^{2}}\Vert \nabla
(T_{k}u^{+})^{m}\Vert _{2}^{2}+\frac{\alpha }{2}\Phi _{k,m}(u^{+})\leq \frac{%
\alpha }{2}|\Omega |+|\Omega |C_{2}\left( \frac{2C_{2}}{\alpha }\right) ^{%
\frac{2m-2}{p}}+\Vert g\Vert _{s}\Vert T_{k}u^{+}\Vert _{s^{\prime
}(2m-1)}^{2m-1}.
\end{equation*}%
Denoting $w=(T_{k}u^{+})^{m}$ we get%
\begin{equation}
\frac{d}{dt}\Phi _{k,m}(u^{+})+\frac{2m-1}{m^{2}}\Vert \nabla w\Vert
_{2}^{2}+\frac{\alpha }{2}\Phi _{k,m}(u^{+})\leq \frac{\alpha }{2}|\Omega
|+|\Omega |C_{2}\left( \frac{2C_{2}}{\alpha }\right) ^{\frac{2m-2}{p}}+\Vert
g\Vert _{s}\Vert w\Vert _{s^{\prime }\frac{2m-1}{m}}^{\frac{2m-1}{m}}.
\label{eq:1}
\end{equation}%
Since $s^{\prime }<\frac{d}{d-2}$ we deduce, by interpolation, that 
\begin{equation*}
\Vert w\Vert _{s^{\prime }\frac{2m-1}{m}}^{\frac{2m-1}{m}}\leq \Vert w\Vert
_{\frac{2d}{d-2}}^{\left( \frac{2m-1}{m}-\frac{1}{s^{\prime }}\right) \frac{%
2d}{d+2}}\Vert w\Vert _{1}^{\frac{2d-(d-2)s^{\prime }\frac{2m-1}{m}}{%
s^{\prime }(d+2)}},
\end{equation*}%
(note that $s^{\prime }\frac{2m-1}{m}<2s^{\prime }<\frac{2d}{d-2}$). Using
the Sobolev embedding $H_{0}^{1}(\Omega)\subset L^{\frac{2d}{d-2}}(\Omega)$ we obtain 
\begin{equation*}
\Vert w\Vert _{s^{\prime }\frac{2m-1}{m}}^{\frac{2m-1}{m}}\leq D\Vert \nabla
w\Vert _{2}^{\left( \frac{2m-1}{m}-\frac{1}{s^{\prime }}\right) \frac{2d}{d+2%
}}\Vert w\Vert _{1}^{\frac{2d-(d-2)s^{\prime }\frac{2m-1}{m}}{s^{\prime
}(d+2)}},
\end{equation*}%
or, after a straightforward calculation, 
\begin{equation*}
\Vert w\Vert _{s^{\prime }\frac{2m-1}{m}}^{\frac{2m-1}{m}}\leq D\Vert \nabla
w\Vert _{2}^{\frac{(2ms^{\prime }-m-s^{\prime })2d}{ms^{\prime }(d+2)}}\Vert
w\Vert _{1}^{\frac{2md-(d-2)s^{\prime }(2m-1)}{ms^{\prime }(d+2)}}.
\end{equation*}%
We use the following inequality valid for $\alpha \in (0,2)$, $\beta >0$ and 
$\epsilon >0$, 
\begin{equation*}
a^{\alpha }b^{\beta }\leq \epsilon ^{\frac{2}{\alpha }}\frac{\alpha }{2}%
a^{2}+\frac{b^{\frac{2\beta }{2-\alpha }}(2-\alpha )}{2\epsilon ^{\frac{2}{%
2-\alpha }}}\leq \epsilon ^{\frac{2}{\alpha }}a^{2}+\frac{b^{\frac{2\beta }{%
2-\alpha }}}{\epsilon ^{\frac{2}{2-\alpha }}},
\end{equation*}%
whence 
\begin{equation*}
\Vert g\Vert _{s}\Vert w\Vert _{s^{\prime }\frac{2m-1}{m}}^{\frac{2m-1}{m}%
}\leq D\Vert g\Vert _{s}\epsilon ^{\frac{ms^{\prime }(d+2)}{(2ms^{\prime
}-m-s^{\prime })d}}\Vert \nabla w\Vert _{2}^{2}+D\Vert g\Vert _{s}\Vert
w\Vert _{1}^{2-\frac{s^{\prime }(d+2)}{m(d-ds^{\prime }+2s^{\prime
})+ds^{\prime }}}\frac{1}{\epsilon ^{\frac{ms^{\prime }(d+2)}{m(2s^{\prime
}-s^{\prime }d+d)+s^{\prime }d}}}.
\end{equation*}%
We need to choose $\epsilon $ such that the constant in front of $\Vert
\nabla w\Vert _{2}^{2}$ is equal to $\frac{1}{2m}$. This is the case if $%
\epsilon =(2m\Vert g\Vert _{s}D)^{-\frac{(2ms^{\prime }-m-s^{\prime })d}{%
ms^{\prime }(d+2)}}$. We arrive at the inequality 
\begin{equation*}
\Vert g\Vert _{s}\Vert w\Vert _{s^{\prime }\frac{2m-1}{m}}^{\frac{2m-1}{m}%
}\leq \frac{1}{2m}\Vert \nabla w\Vert _{2}^{2}+D\Vert g\Vert _{s}\Vert
w\Vert _{1}^{2-\frac{s^{\prime }(d+2)}{m(d-ds^{\prime }+2s^{\prime
})+ds^{\prime }}}(2m\Vert g\Vert _{s}D)^{\frac{m(2s^{\prime }d-d)-s^{\prime
}d}{m(2s^{\prime }-s^{\prime }d+d)+s^{\prime }d}}.
\end{equation*}%
Using the fact that the exponent of $\Vert w\Vert _{1}$ is nonnegative, and
substituting the last bound in \eqref{eq:1}, we obtain 
\begin{align*}
& \frac{d}{dt}\Phi _{k,m}(u^{+})+\frac{3m-2}{2m^{2}}\Vert \nabla w\Vert
_{2}^{2}+\frac{\alpha }{2}\Phi _{k,m}(u^{+}) \\
& \ \ \leq \frac{\alpha }{2}|\Omega |+|\Omega |C_{2}\left( \frac{2C_{2}}{%
\alpha }\right) ^{\frac{2m-2}{p}}+D\Vert g\Vert _{s}(1+\Vert w\Vert
_{1}^{2})(2m\Vert g\Vert _{s}D)^{\frac{m(2s^{\prime }d-d)-s^{\prime }d}{%
m(2s^{\prime }-s^{\prime }d+d)+s^{\prime }d}}.
\end{align*}%
Multiplying the last equation by $2m$ it follows that 
\begin{equation*}
\frac{d}{dt}2m\Phi _{k,m}(u^{+})+\frac{\alpha }{2}2m\Phi _{k,m}(u^{+})+\Vert
\nabla w\Vert _{2}^{2}\leq DmD^{m}+D\left( 1+\Vert w\Vert _{1}^{2}\right) m^{%
\frac{m(s^{\prime }d+2s^{\prime })}{m(2s^{\prime }-s^{\prime }d+d)+s^{\prime
}d}}.
\end{equation*}%
Using the fact that $m\geq 1$ we deduce, changing the constant $D$ if
necessary that 
\begin{equation}
\frac{d}{dt}2m\Phi _{k,m}(u^{+})+\frac{\alpha }{2}2m\Phi _{k,m}(u^{+})+\Vert
\nabla w\Vert _{2}^{2}\leq D^{m}+D\Vert w\Vert _{1}^{2}m^{D}.  \label{eq:it}
\end{equation}%
Dropping the term with $\nabla w$ and multiplying by the integrating factor $%
e^{\frac{\alpha }{2}t}$ we obtain 
\begin{equation*}
\frac{d}{dt}(2m\Phi _{k,m}(u^{+})e^{\frac{\alpha }{2}t})\leq D^{m}e^{\frac{%
\alpha }{2}t}+D\Vert w\Vert _{1}^{2}m^{D}e^{\frac{\alpha }{2}t}.
\end{equation*}%
Now, integrating from $s$ to $t_{2}$ with $t_{1}\leq s\leq t_{2}$ it follows
that 
\begin{equation*}
2m\Phi _{k,m}(u^{+}(t_{2}))\leq 2m\Phi _{k,m}(u^{+}(s))e^{\frac{\alpha }{2}%
(s-t_{2})}+\frac{2}{\alpha }D^{m}+\frac{2D}{\alpha }m^{D}\sup_{t\in \lbrack
t_{1},t_{2}]}\Vert w(t)\Vert _{1}^{2}.
\end{equation*}%
We can rewrite, changing the constant $D$ if necessary, as 
\begin{equation*}
2m\Phi _{k,m}(u^{+}(t_{2}))\leq 2m\Phi
_{k,m}(u^{+}(s))+D^{m}+Dm^{D}\sup_{t\in \lbrack t_{1},t_{2}]}\Vert w(t)\Vert
_{1}^{2},
\end{equation*}%
\begin{equation*}
2m\Phi _{k,m}(u^{+}(t_{2}))\leq 2m\Phi
_{k,m}(u^{+}(s))+D^{m}+Dm^{D}\sup_{t\in \lbrack t_{1},t_{2}]}\Vert
T_{k}u^{+}(t)\Vert _{m}^{2m}.
\end{equation*}%
By Lemma \ref{lem:Phi} 
\begin{equation}
\Vert T_{k}u^{+}(t_{2})\Vert _{2m}^{2m}\leq 2m\Phi
_{k,m}(u^{+}(s))+D^{m}+Dm^{D}\sup_{t\in \lbrack t_{1},t_{2}]}\Vert
u^{+}(t)\Vert _{m}^{2m}.  \label{eq:basic}
\end{equation}%
%
%
%
Moreover, from \eqref{eq:it}, by the Sobolev embedding theorem 
\begin{equation*}
\int_{t_{1}}^{t_{2}}\Vert w(s)\Vert _{\frac{2d}{d-2}}^{2}\,ds\leq D2m\Phi
_{k,m}(u^{+}(t_{1}))+D^{m}(t_{2}-t_{1})+D(t_{2}-t_{1})\sup_{t\in \lbrack
t_{1},t_{2}]}\Vert w(t)\Vert _{1}^{2}m^{D}.
\end{equation*}%
After dividing by $t_{2}-t_{1}$ and using the definition of $w$ 
\begin{equation*}
\frac{1}{t_{2}-t_{1}}\int_{t_{1}}^{t_{2}}\Vert T_{k}u^{+}(s)\Vert _{m\frac{2d%
}{d-2}}^{2m}\,ds\leq \frac{D2m\Phi _{k,m}(u^{+}(t_{1}))}{t_{2}-t_{1}}%
+D^{m}+\sup_{t\in \lbrack t_{1},t_{2}]}\Vert T_{k}u^{+}(t)\Vert
_{m}^{2m}Dm^{D}.
\end{equation*}%
This yields, by Lemma \ref{lem:Phi}, changing $m$ to $\widetilde{m}=m\frac{d%
}{d-2}$ and denoting the new symbol again by $m$, noting that \eqref{eq:it}
was valid for $m\geq 1$ we obtain for $m\geq 2$ and $d\geq 4$ that 
\begin{equation}
\frac{1}{t_{2}-t_{1}}\int_{t_{1}}^{t_{2}}\Vert T_{k}u^{+}(s)\Vert _{2m}^{2m%
\frac{d-2}{d}}\,ds\leq \frac{D}{t_{2}-t_{1}}\Vert u^{+}(t_{1})\Vert _{2m%
\frac{d-2}{d}}^{2m\frac{d-2}{d}}+D^{m}+\sup_{t\in \lbrack t_{1},t_{2}]}\Vert
u^{+}(t)\Vert _{m\frac{d-2}{d}}^{2m\frac{d-2}{d}}Dm^{D}.  \label{case1}
\end{equation}%
We require that $m\geq 2$. However, for $d=3$ the above inequality is valid
for $m\geq 3$, for $d\geq 4$ it is valid for $m\geq \frac{d}{d-2}$ (and, in
particular, for $m\geq 2$). To get the above inequality for $d=3$, and $m\in
\lbrack 2,3)$, in place of the embedding $H^{1}(\Omega)\subset L^{6}(\Omega)$ we use the
embedding $H^{1}(\Omega)\subset L^{4}(\Omega)$ in \eqref{eq:it}, whence, by Lemma \ref%
{lem:Phi}, 
\begin{equation*}
\frac{1}{t_{2}-t_{1}}\int_{t_{1}}^{t_{2}}\Vert T_{k}u^{+}(s)\Vert
_{4m}^{2m}\,ds\leq \frac{D}{t_{2}-t_{1}}\Vert u^{+}(t_{1})\Vert
_{2m}^{2m}+D^{m}+D\sup_{t\in \lbrack t_{1},t_{1}]}\Vert T_{k}u^{+}(t)\Vert
_{m}^{2m}m^{D}.
\end{equation*}%
We change $m$ to $\widetilde{m}=2m$ and denote the new symbol again by $m$,
whence 
\begin{equation*}
\frac{1}{t_{2}-t_{1}}\int_{t_{1}}^{t_{2}}\Vert T_{k}u^{+}(s)\Vert
_{2m}^{m}\,ds\leq \frac{D}{t_{2}-t_{1}}\Vert u^{+}(t_{1})\Vert
_{m}^{m}+D^{m}+D\sup_{t\in \lbrack t_{1},t_{1}]}\Vert T_{k}u^{+}(t)\Vert _{%
\frac{m}{2}}^{m}m^{D}.
\end{equation*}%
Using the fact that $\Vert v\Vert _{\frac{m}{2}}^{m}\leq |\Omega |\Vert
v\Vert _{m}^{m}$ for $m\geq 2$ we deduce 
\begin{equation}
\frac{1}{t_{2}-t_{1}}\int_{t_{1}}^{t_{2}}\Vert T_{k}u^{+}(s)\Vert
_{2m}^{m}\,ds\leq \frac{D}{t_{2}-t_{1}}\Vert u^{+}(t_{1})\Vert
_{m}^{m}+D^{m}+D\sup_{t\in \lbrack t_{1},t_{1}]}\Vert u^{+}(t)\Vert
_{m}^{m}m^{D}.  \label{case2}
\end{equation}%
Note that the right-hand sides both in \eqref{case1} and \eqref{case2} are
independent on the truncation level $k$. Hence, we are in position to use
the monotone convergence theorem, and from \eqref{case1} we obtain if $d\geq
4$ or $d=3$ and $m\geq 3$ that%
\begin{equation}
\frac{1}{t_{2}-t_{1}}\int_{t_{1}}^{t_{2}}\Vert u^{+}(s)\Vert _{2m}^{2m\frac{%
d-2}{d}}\,ds\leq \frac{D}{t_{2}-t_{1}}\Vert u^{+}(t_{1})\Vert _{2m\frac{d-2}{%
d}}^{2m\frac{d-2}{d}}+D^{m}+\sup_{t\in \lbrack t_{1},t_{2}]}\Vert
u^{+}(t)\Vert _{m\frac{d-2}{d}}^{2m\frac{d-2}{d}}Dm^{D}.  \label{case1_b}
\end{equation}%
On the other hand we obtain from \eqref{case2} if $d=3$ and $m\in \lbrack
2,3)$ 
\begin{equation}
\frac{1}{t_{2}-t_{1}}\int_{t_{1}}^{t_{2}}\Vert u^{+}(s)\Vert
_{2m}^{m}\,ds\leq \frac{D}{t_{2}-t_{1}}\Vert u^{+}(t_{1})\Vert
_{m}^{m}+D^{m}+Dm^{D}\sup_{t\in \lbrack t_{1},t_{1}]}\Vert u^{+}(t)\Vert
_{m}^{m}.  \label{case2_b}
\end{equation}%
In the case of \eqref{case1_b} raise \eqref{eq:basic} to the power $\frac{d-2%
}{d}$, whence 
\begin{equation*}
\Vert T_{k}u^{+}(t_{2})\Vert _{2m}^{2m\frac{d-2}{d}}\leq (2m\Phi
_{k,m}(u^{+}(s)))^{\frac{d-2}{d}}+D^{m}+Dm^{D}\sup_{t\in \lbrack
t_{1},t_{2}]}\Vert u^{+}(t)\Vert _{m}^{2m\frac{d-2}{d}},
\end{equation*}%
for $s\in \lbrack t_{1},t_{2}]$. After integrating with respect to $s$ from $%
t_{1}$ to $t_{2}$, we can use \eqref{case1_b}, whence 
\begin{align*}
& \Vert T_{k}u^{+}(t_{2})\Vert _{2m}^{2m\frac{d-2}{d}}\leq D^{m} \\
& \qquad \qquad +\frac{D}{t_{2}-t_{1}}\Vert u^{+}(t_{1})\Vert _{2m\frac{d-2}{%
d}}^{2m\frac{d-2}{d}}+Dm^{D}\sup_{t\in \lbrack t_{1},t_{2}]}\Vert
u^{+}(t)\Vert _{m\frac{d-2}{d}}^{2m\frac{d-2}{d}}+Dm^{D}\sup_{t\in \lbrack
t_{1},t_{2}]}\Vert u^{+}(t)\Vert _{m}^{2m\frac{d-2}{d}},
\end{align*}%
Taking the $\frac{d}{2\left( d-2\right) }$-root we have 
\begin{align*}
& \Vert T_{k}u^{+}(t_{2})\Vert _{2m}^{m}\leq D^{m} \\
& \qquad \qquad +\frac{D}{\left( t_{2}-t_{1}\right) ^{\frac{d}{2\left(
d-2\right) }}}\Vert u^{+}(t_{1})\Vert _{2m\frac{d-2}{d}}^{m}+Dm^{D}\sup_{t%
\in \lbrack t_{1},t_{2}]}\Vert u^{+}(t)\Vert _{m\frac{d-2}{d}%
}^{m}+Dm^{D}\sup_{t\in \lbrack t_{1},t_{2}]}\Vert u^{+}(t)\Vert _{m}^{m}.
\end{align*}%
Pick $t_{2}-t_{1}=\tau \geq \frac{\delta }{m^{D}}$ with arbitrarily small $%
\delta >0$, whence 
\begin{align*}
& \Vert T_{k}u^{+}(t_{1}+\tau )\Vert _{2m}^{m}\leq D^{m} \\
& \qquad \qquad +D(\delta )m^{D}\left( \Vert u^{+}(t_{1})\Vert _{2m\frac{d-2%
}{d}}^{m}+\sup_{t\in \lbrack t_{1},t_{1}+\tau ]}\Vert u^{+}(t)\Vert _{m\frac{%
d-2}{d}}^{m}+\sup_{t\in \lbrack t_{1},t_{1}+\tau ]}\Vert u^{+}(t)\Vert
_{m}^{m}\right) ,
\end{align*}%
or, using the fact that $\Vert v\Vert _{m\frac{d-2}{d}}^{m}\leq |\Omega |^{%
\frac{d}{2\left( d-2\right) }}\Vert v\Vert _{2m\frac{d-2}{d}}^{m}$, 
\begin{equation}
\Vert T_{k}u^{+}(t_{1}+\tau )\Vert _{2m}^{m}\leq D^{m}+D(\delta )m^{D}\left(
\sup_{t\in \lbrack t_{1},t_{1}+\tau ]}\Vert u^{+}(t)\Vert _{2m\frac{d-2}{d}%
}^{m}+\sup_{t\in \lbrack t_{1},t_{1}+\tau ]}\Vert u^{+}(t)\Vert
_{m}^{m}\right) .  \label{case11}
\end{equation}%
In case of \eqref{case2_b} we raise \eqref{eq:basic} to the power $\frac{1}{2%
}$, whence 
\begin{equation*}
\Vert T_{k}u^{+}(t_{2})\Vert _{2m}^{m}\leq (2m\Phi _{k,m}(u^{+}(s)))^{\frac{1%
}{2}}+D^{m}+Dm^{D}\sup_{t\in \lbrack t_{1},t_{2}]}\Vert u^{+}(t)\Vert
_{m}^{m}.
\end{equation*}%
Integrate with respect to $s$ from $t_{1}$ to $t_{2}$ and use \eqref{case2_b}%
, whence 
\begin{equation*}
\Vert T_{k}u^{+}(t_{2})\Vert _{2m}^{m}\leq D^{m}+\frac{D}{t_{2}-t_{1}}\Vert
u^{+}(t_{1})\Vert _{m}^{m}+Dm^{D}\sup_{t\in \lbrack t_{1},t_{2}]}\Vert
u^{+}(t)\Vert _{m}^{m}.
\end{equation*}%
Again pick $t_{2}-t_{1}=\tau \geq \frac{\delta }{m^{D}}$ with arbitrarily
small $\delta $ whence 
\begin{equation}
\Vert T_{k}u^{+}(t_{2})\Vert _{2m}^{m}\leq D^{m}+D(\delta )m^{D}\sup_{t\in
\lbrack t_{1},t_{2}]}\Vert u^{+}(t)\Vert _{m}^{m}.  \label{case22}
\end{equation}%
We consider three cases.

\textbf{Case 1.} If $d\geq 4$ then in \eqref{case11} the term with the norm $%
2m\frac{d-2}{d}$ always dominates in the right-hand side, and 
\begin{equation*}
\Vert T_{k}u^{+}(t_{1}+\tau )\Vert _{2m}^{m}\leq D^{m}+D(\delta
)m^{D}\sup_{t\in \lbrack t_{1},t_{1}+\tau ]}\Vert u^{+}(t)\Vert _{2m\frac{d-2%
}{d}}^{m}.
\end{equation*}%
Using the monotone convergence theorem we pass with $k$ to infinity and we
deduce that, provided the right-hand side is finite, 
\begin{equation*}
\max \left\{ \Vert u^{+}(t_{1}+\tau )\Vert _{2m},D\right\} ^{m}\leq D(\delta
)m^{D}\max \left\{ \sup_{t\in \lbrack t_{1},t_{1}+\tau ]}\Vert u^{+}(t)\Vert
_{2m\frac{d-2}{d}},D\right\} ^{m},
\end{equation*}%
where the constants $D$ under $\max $ in the left and right-hand side
coincide. Taking the $m$-th root we obtain 
\begin{equation*}
\max \left\{ \Vert u^{+}(t_{1}+\tau )\Vert _{2m},D\right\} \leq D(\delta )^{%
\frac{1}{m}}m^{\frac{D}{m}}\max \left\{ \sup_{t\in \lbrack t_{1},t_{1}+\tau
]}\Vert u^{+}(t)\Vert _{2m\frac{d-2}{d}},D\right\} .
\end{equation*}%
The above inequality is valid for $m\geq \frac{d}{d-2}$. Hence we can
introduce the new variable $\widetilde{m}=2m\frac{d-2}{d}$. Denoting this
new variable by $m$ we get for $m\geq 2$ that%
\begin{equation*}
\max \left\{ \Vert u^{+}(t_{1}+\tau )\Vert _{\frac{d}{d-2}m},D\right\} \leq
D(\delta )^{\frac{1}{m}}m^{\frac{D}{m}}\max \left\{ \sup_{t\in \lbrack
t_{1},t_{1}+\tau ]}\Vert u^{+}(t)\Vert _{m},D\right\} .
\end{equation*}%
\textbf{Case 2.} If $d=3$ and $m\geq 3$ then the term with the $m$-th norm
dominates in \eqref{case11}. Proceeding analogously as in case 1 we obtain 
\begin{equation*}
\max \{\Vert u^{+}(t_{1}+\tau )\Vert _{2m},D\}^{m}\leq D(\delta )m^{D}\max
\left\{ \sup_{t\in \lbrack t_{1},t_{1}+\tau ]}\Vert u^{+}(t)\Vert
_{m},D\right\} ^{m},
\end{equation*}%
whence, after taking the $m$-th root it follows that 
\begin{equation*}
\max \left\{ \Vert u^{+}(t_{1}+\tau )\Vert _{2m},D\right\} \leq D(\delta )^{%
\frac{1}{m}}m^{\frac{D}{m}}\max \left\{ \sup_{t\in \lbrack t_{1},t_{1}+\tau
]}\Vert u^{+}(t)\Vert _{m},D\right\} .
\end{equation*}

\textbf{Case 3.} If $d=3$ and $m\in [2,3)$ we use \eqref{case22} in place of %
\eqref{case11}, which leads us to exactly the same assertion as Case 2.

We conclude the proof by the discussion of required modifications for the case $d\in \{1,2\}$. 

\textbf{The case of $d\in \{1,2\}$.} If $d=2$ then in \eqref{eq:1} we need to use
the interpolation inequality 
\begin{equation*}
\Vert w\Vert _{s^{\prime }\frac{2m-1}{m}}^{\frac{2m-1}{m}}\leq \Vert w\Vert
_{1}^{\frac{m+1}{m(3s^{\prime }-1)}}\Vert w\Vert _{3s^{\prime }}^{\frac{%
3(2s^{\prime }m-m-s^{\prime })}{m(3s^{\prime }-1)}}\leq D\Vert w\Vert _{1}^{%
\frac{m+1}{m(3s^{\prime }-1)}}\Vert \nabla w\Vert _{2}^{\frac{3(2s^{\prime
}m-m-s^{\prime })}{m(3s^{\prime }-1)}}.
\end{equation*}%
Proceeding analogously as in the argument that leads to \eqref{eq:it} we
obtain 
\begin{equation*}
\Vert g\Vert _{s}\Vert w\Vert _{s^{\prime }\frac{2m-1}{m}}^{\frac{2m-1}{m}%
}\leq \frac{1}{2m}\Vert \nabla w\Vert _{2}^{2}+D\Vert g\Vert _{s}\Vert
w\Vert _{1}^{\frac{2(m+1)}{m+3s^{\prime }}}(2m\Vert g\Vert _{s}D)^{\frac{%
3(2s^{\prime }m-m-s^{\prime })}{m+3s^{\prime }}},
\end{equation*}%
which leads exactly to \eqref{eq:basic}. The Sobolev embedding $H^{1}(\Omega)\subset
L^{4}(\Omega)$ leads us to \eqref{case2_b} in case $d=2$. Then, the argument to get %
\eqref{case22} follows the lines of the argument in the above proof and we
get the assertion of the lemma with $A=2$. On the other hand, if $d=1$, then
we can take $g\in L^{1}(\Omega)$ and we estimate the term $(g,(T_{k}u^{+})^{2m-1})$
as follows 
\begin{align*}
& (g,(T_{k}u^{+})^{2m-1})\leq \Vert g\Vert _{1}\Vert T_{k}u^{+}\Vert
_{\infty }^{2m-1}=\Vert g\Vert _{1}\Vert w\Vert _{\infty }^{\frac{2m-1}{m}%
}\leq C\Vert g\Vert _{1}\Vert \nabla w\Vert _{2}^{\frac{2m-1}{2m}}\Vert
w\Vert _{1}^{\frac{2m-1}{2m}} \\
& \ \ \leq \frac{1}{2m}\Vert \nabla w\Vert _{2}^{2}+D(1+\Vert w\Vert
_{1}^{2})m,
\end{align*}%
which leads us to \eqref{eq:it} and \eqref{eq:basic}. The rest of the proof
follows the lines of the case $d=2$ which yields the assertion of the lemma
with $A=2$.
\end{proof}

\bigskip

Using the iteration procedure we will iteratively increase $m$ by its
multiplication by a constant greater than one in each step at the cost of
short increase of initial time.

We can rewrite the assertion of Lemma \ref{lem:33} as 
\begin{equation*}
\max \{\Vert u^{+}(t_{1}+\tau _{1})\Vert _{Am},D\}\leq \alpha (m)\max
\left\{ \sup_{t\geq t_{1}}\Vert u^{+}(t)\Vert _{m},D\right\} \ \text{for}\
\tau _{1}\geq \frac{\delta }{m^{D}}.
\end{equation*}%
Writing this inequality for $A^{i}m$ in place of $m$ we obtain 
\begin{equation*}
\max \{\Vert u^{+}(t_{2}+\tau _{2})\Vert _{A^{i+1}m},D\}\leq \alpha
(A^{i}m)\max \left\{ \sup_{t\geq t_{2}}\Vert u^{+}(t)\Vert
_{A^{i}m},D\right\} \ \text{for}\ \tau _{2}\geq \frac{\delta }{A^{iD}m^{D}}.
\end{equation*}%
Combining the above inequalities with each other for $i=1$ to $k-1$ for some
given natural $k$ we obtain 
\begin{equation*}
\max \{\Vert u^{+}(t_{1}+\tau )\Vert _{A^{k}m},D\}\leq \alpha
(A^{k-1}m)\cdot \ldots \cdot \alpha (Am)\alpha (m)\max \left\{ \sup_{t\geq
t_{1}}\Vert u^{+}(t_{1})\Vert _{m},D\right\} 
\end{equation*}%
for$\ \tau \geq \sum_{i=0}^{k-1}\frac{\delta }{A^{iD}m^{D}}.$ Now 
\begin{align*}
& \alpha (A^{k-1}m)\cdot \ldots \cdot \alpha (Am)\alpha (m) \\
& \ =D(\delta )^{\frac{1}{A^{k-1}m}+\frac{1}{A^{k-2}m}+\ldots +\frac{1}{Am}+%
\frac{1}{m}}(A^{k-1}m)^{\frac{D}{A^{k-1}m}}(A^{k-2}m)^{\frac{D}{A^{k-2}m}%
}\ldots (Am)^{\frac{D}{Am}}m^{\frac{D}{m}}.
\end{align*}%
We calculate that 
\begin{equation*}
D(\delta )^{\frac{1}{A^{k-1}m}+\frac{1}{A^{k-2}m}+\ldots +\frac{1}{Am}+\frac{%
1}{m}}\leq D(\delta )^{\frac{1}{m}\frac{A}{A-1}},
\end{equation*}%
\begin{align*}
& \ln \left( (A^{k-1}m)^{\frac{D}{A^{k-1}m}}(A^{k-2}m)^{\frac{D}{A^{k-2}m}%
}\ldots (Am)^{\frac{D}{Am}}m^{\frac{D}{m}}\right) =\sum_{i=0}^{k-1}\frac{D}{%
A^{i}m}(\ln m+i\ln A) \\
& \ \ \leq \frac{D\ln m}{m}\frac{A}{A-1}+\frac{D\ln A}{m}\frac{A}{(A-1)^{2}}.
\end{align*}%
Hence 
\begin{equation*}
\alpha (A^{k-1}m)\cdot \ldots \cdot \alpha (Am)\alpha (m)\leq D(\delta )^{%
\frac{1}{m}\frac{A}{A-1}}m^{\frac{D}{m}\frac{A}{A-1}}A^{\frac{D}{m}\frac{A}{%
(A-1)^{2}}}.
\end{equation*}%
This means that 
\begin{equation*}
\max \{\Vert u^{+}(t_{1}+\tau )\Vert _{A^{k}m},D\}\leq D(\delta )^{\frac{1}{m%
}\frac{A}{A-1}}m^{\frac{D}{m}\frac{A}{A-1}}A^{\frac{D}{m}\frac{A}{(A-1)^{2}}%
}\max \left\{ \sup_{t\geq t_{1}}\Vert u^{+}(t_{1})\Vert _{m},D\right\} \ 
\text{for}\ \tau \geq \frac{\delta }{m^{D}}\frac{A^{D}}{A^{D}-1}.
\end{equation*}%
In particular we take $m=2$ whence, for every $k\in \mathbb{N}$ we obtain
(after changing $D(\delta )$) that 
\begin{equation*}
\max \{\Vert u^{+}(t_{1}+\tau )\Vert _{2A^{k}},D\}\leq D(\delta )\max
\left\{ \sup_{t\geq t_{1}}\Vert u^{+}(t_{1})\Vert _{2},D\right\} \ \text{for}%
\ \tau \geq \delta \frac{A^{D}}{2^{D}(A^{D}-1)}.
\end{equation*}%
We can pass with $k$ to infinity whence 
\begin{equation*}
\Vert u^{+}(t_{1}+\tau )\Vert _{\infty }\leq D(\delta )\left( \sup_{t\geq
t_{1}}\Vert u^{+}(t_{1})\Vert _{2}+1\right) \ \text{for}\ \tau \geq \delta 
\frac{A^{D}}{2^{D}(A^{D}-1)}.
\end{equation*}%
The argument for the negative part $u^{-}$ is analogous and we skip it. We
have hence proved the following result.

\begin{theorem}
\label{thm:linfty} If $u$ is a weak solution of \eqref{eq:basic} with $%
\sup_{s\geq t_1}\|u(s)\|_2 < \infty$, then for every $\tau > 0$ there exists
a constant $D(\tau)$ such that for every $t\geq t_1+\tau$ we have $u(t)\in
L^\infty(\Omega)$ and 
\begin{equation*}
\|u(t)\|_\infty \leq D(\tau)\left(\sup_{s\geq t_1}\|u(s)\|_2+1\right).
\end{equation*}
\end{theorem}

\section{Existence and structure of the global attractor for weak and
regular solutions}

In this section we assume the conditions \eqref{Contf}-\eqref{Diss} and (\ref%
{g2}).

\begin{lemma}
\label{WeakReg}Any weak solution of problem (\ref{Eq}) is a regular solution.
\end{lemma}

\begin{proof}
In view of Theorem \ref{LinfEst} every weak solution $u$ satisfies that 
\begin{equation*}
u\in L^{\infty }(\varepsilon ,T;L^{\infty }(\Omega ))\text{ for all }%
0<\varepsilon <T.
\end{equation*}%
Then the function 
\begin{equation*}
d(t)=f(u(t))+g
\end{equation*}%
belongs to $L^{\infty }(\varepsilon ,T;L^{2}(\Omega ))$. Also, since $u\in
L^{2}(0,T;H_{0}^{1}(\Omega ))\cap C([0,T];L^2(\Omega))$, there exists $\varepsilon ^{\prime }<\varepsilon 
$ such that $u\left( \varepsilon ^{\prime }\right) \in H_{0}^{1}(\Omega )$. 
Consider the problem 
\begin{equation*}
\left\{ 
\begin{array}{l}
v_{t}-\Delta v=d(t),\quad x\in \Omega ,\ t>\varepsilon ^{\prime }, \\ 
v|_{\partial \Omega }=0, \\ 
v(\varepsilon ^{\prime })=u(\varepsilon ^{\prime }),%
\end{array}%
\right.
\end{equation*}%
which has a unique weak solution satisfying moreover that 
\begin{equation*}
v\in C([\varepsilon ^{\prime },T],H_{0}^{1}(\Omega ))\cap L^{2}(\varepsilon
^{\prime },T;H^{2}(\Omega ))
\end{equation*}%
(see e.g. \cite[p.70]{Temam}). Hence, $v_{t}\in L^{2}(\varepsilon ^{\prime
},T;L^{2}(\Omega ))$, and, since $u\left( \text{\textperiodcentered }\right)
=v\left( \text{\textperiodcentered }\right) $ on $[\varepsilon ^{\prime
},\infty )$, it follows that $u$ is a regular solution to problem (\ref{Eq}).
\end{proof}

\begin{corollary}
\label{Complete}$\mathbb{K}_{r}=\mathbb{K}$, $\mathbb{F}_{r}=\mathbb{F}.$
\end{corollary}

\bigskip

This lemma implies that the semiflows $G$ and $G_{r}$ coincide, so $G_{r}$
is a strict multivalued semiflow. We know that $G$ has the global compact
invariant connected attractor $\mathcal{A}$. We observe also that now the
results in \cite{KKV14} about the global attractor for $G_{r}$ are true
without condition (\ref{Condp}), although the assumption on the function $g$
is stronger for $d\geq 4$. Hence, $G_{r}$ posseses the global attractor $%
\mathcal{A}_{r},$ which is a compact set in $H_{0}^{1}\left( \Omega \right) $%
. Moreover,%
\begin{equation*}
\textrm{dist}_{H_{0}^{1}(\Omega )}(G_{r}(t,B),\mathcal{A}_{r})\rightarrow 0\text{ as }%
t\rightarrow +\infty ,
\end{equation*}%
for any bounded set $B\subset L^2(\Omega)$, and $\mathcal{A}_{r}$ satisfies (\ref{Str}).

\begin{lemma}
$\mathcal{A}_{r}=\mathcal{A}.$ Thus, $\mathcal{A}_{r}$ is connected and
invariant.
\end{lemma}

\begin{proof}
For any $\delta >0$ there is $T(\delta )>0$ such that 
\begin{equation*}
\mathcal{A}\subset G(T(\delta ),\mathcal{A})=G_{r}(T(\delta ),\mathcal{A}
)\subset O_{\delta }(\mathcal{A}_{r}),
\end{equation*}
so that $\mathcal{A\subset A}_{r}$. The converse is proved in the same way.
\end{proof}

\bigskip

Denote by $\mathbb{K}^{1}$ the set of all complete trajectories of $K^{+}$
which are bounded in $H_{0}^{1}(\Omega )$. Since $\mathbb{K}_{r}=\mathbb{K}%
_{r}^{1}$, Lemma \ref{WeakReg} and Corollary \ref{Complete} imply that 
\begin{equation*}
\mathbb{K}^{1}=\mathbb{K}_{r}^{1}=\mathbb{K}_{r}=\mathbb{K}\text{.}
\end{equation*}%
Then%
\begin{eqnarray*}
\mathcal{A} &=&\{\gamma \left( 0\right) :\gamma \left(\text{\textperiodcentered}\right) \in \mathbb{K}\}=\{\gamma \left( 0\right)
:\gamma \left( \text{\textperiodcentered }\right) \in \mathbb{K}^{1}\} \\
&=&\cup _{t\in \mathbb{R}}\{\gamma \left( t\right) :\gamma \left(\text{\textperiodcentered }\right) \in \mathbb{K}\}=\cup _{t\in \mathbb{R}
}\{\gamma \left( t\right) :\gamma \left( \text{\textperiodcentered }\right)
\in \mathbb{K}^{1}\}.
\end{eqnarray*}

\begin{lemma}
\label{BoundedLinf}The attractor $\mathcal{A}=\mathcal{A}_{r}$ is a bounded
set in $L^{\infty }(\Omega )$ and a compact set in $H_{0}^{1}(\Omega )$.
\end{lemma}

\begin{proof}
Since $\mathcal{A}$ is bounded in $L^{2}(\Omega )$ and invariant, we can
take $\tau =1$ in Theorem \ref{thm:linfty} to obtain the required
boundedness in $L^{\infty }(\Omega )$. Compactness in $H_{0}^{1}(\Omega )$
follows from the fact that $\mathcal{A}_{r}$ is compact in $H_{0}^{1}(\Omega
)$.
\end{proof}

\begin{remark}
\label{BoundedLinf2}The boundedness in $L^{\infty }(\Omega )$ is true under
conditon (\ref{g}) as well.
\end{remark}

Denote%
\begin{equation*}
\begin{array}{c}
M^{-}(\mathfrak{R})=\left\{ z\,:\,\exists \gamma (\cdot )\in \mathbb{K},\,\
\gamma (0)=z,\,\,\,\ \mathrm{dist}_{L^{2}(\Omega )}(\gamma (t),\mathfrak{R}
)\rightarrow 0,\,\ t\rightarrow +\infty \right\} , \\ 
M^{+}(\mathfrak{R})=\left\{ z\,:\,\exists \gamma (\cdot )\in \mathbb{F},\,\
\gamma (0)=z,\,\,\,\ \mathrm{dist}_{L^{2}(\Omega )}(\gamma (t),\mathfrak{R}
)\rightarrow 0,\,\ t\rightarrow -\infty \right\} .%
\end{array}%
\end{equation*}%
It follows from Lemma \ref{Complete} that $M_{r}^{-}(\mathfrak{R})=M^{-}(%
\mathfrak{R}),\ M_{r}^{+}(\mathfrak{R})=M^{+}(\mathfrak{R}).$ Hence, by (\ref%
{Str}) we obtain the characterization of the global attractor $\mathcal{A}$.

\begin{theorem}
$\mathcal{A=}M^{-}(\mathfrak{R})=M^{+}(\mathfrak{R}).$
\end{theorem}

\bigskip

In this way, we have obtained the characterization of the global attractor
for weak solutions for any constant $p\geq 2$ and arbitrary $g\in
L^{2}\left( \Omega \right) $ when $d\leq 3$, improving the results in \cite%
{KKV15}. Also, we have improved the results in \cite{KKV14} concerning the
global attractor for regular solutions by removing condition (\ref{Condp}).
When $d\geq 4$, we have obtained the same results but with a stronger
restriction on the function $g,$ namely that $g\in L^{s}(\Omega )$ for some $%
s>\frac{d}{2}$.

\section{Uniqueness and estimate of the fractal dimension}
In this section under one-sided local Lipschitz condition we obtain the uniqueness result away from the initial data, and the $L^2$ Lipschitz continuous dependence on the initial data on the attractor. Then, reinforcing this assumption 
to (two-sided) local Lipschitz condition we obtain the upper estimates on the global attractor fractal dimension. Throughout this section we always assume conditions (\ref{Contf})-(\ref{g}). 

\subsection{Uniqueness and continuity on attractor.} Using the previous estimates in the space $L^{\infty }(\Omega )$ we obtain uniqueness of weak solutions after an arbitrary small
time if we assume additionally that the function $f$ is locally Lipschitz.
We obtain uniqueness of solutions (and single-valuedness of the semiflow) on the global attractor which allows us to deduce suitable estimates of its the fractal dimension in terms of the
space dimension $d.$

For the results of this subsection, in addition to conditions (\ref{Contf})-(\ref{g}), we require that $%
f $ is locally one-sided Lipschitz, that is, for every $R>0$ there exists a constant $L(R)>0$ such that%
\begin{equation}
\left( f(u)-f(v) \right) \cdot (u-v) \geq -L(R)| u-v|^2 \text{ if 
}\left\vert u\right\vert ,\left\vert v\right\vert \leq R.  \label{Lip}
\end{equation}
This condition means, that the rate at which the function $f$ is increasing may be arbitrary, but if $f$ is decreasing, the rate with which the decrease occurs is bounded by $L(R)$, the one-sided Lipschitz constant. 
\begin{theorem}
\label{Unique}Let $u($\textperiodcentered $)$ be an arbitrary weak solution
to problem (\ref{Eq}) with $u\left( 0\right) \in L^{2}(\Omega )$. Then for
any $\tau >0$ the function $v($\textperiodcentered $)=u($\textperiodcentered 
$+\tau )$ is the unique weak solution to problem (\ref{Eq}) with $v\left(
0\right) =u(\tau ).$
\end{theorem}

\begin{proof}
It is well known \cite[p.625]{KV06} that 
\begin{equation*}
\left\Vert u(t)\right\Vert _{2}\leq \left\Vert u(0)\right\Vert
_{2}e^{-\alpha t}+K
\end{equation*}%
for some constants $\alpha ,K>0.$ Hence, $\sup_{t\geq 0}\left\Vert
u(t)\right\Vert _{2}<\infty $.

In view of Theorem \ref{thm:linfty} for any $\tau >0$ there is $D\left( \tau
\right) >0$ such that 
\begin{equation*}
\Vert u(t)\Vert _{\infty }\leq D(\tau )\left( \sup_{s\geq 0}\Vert u(s)\Vert
_{2}+1\right) =R(\tau )\text{ if }t\geq \tau .
\end{equation*}

By the  translation property \cite[Lemma 3]{KV06}, the translated function $v($\textperiodcentered $%
)=u($\textperiodcentered $+\tau )$ is a weak solution to problem (\ref{Eq})
with $v\left( 0\right) =u(\tau )$. Let $\overline{v}($\textperiodcentered $)$
be another solution to problem (\ref{Eq}) with $v\left( 0\right) =u(\tau )$%
. By concatenation, \cite[Lemma 3]{KV06} the function 
\begin{equation*}
\overline{u}\left( t\right) =\left\{ 
\begin{array}{c}
u(t)\text{ if }0\leq t\leq \tau , \\ 
\overline{v}(t-\tau )\text{ if }t\geq \tau ,%
\end{array}%
\right. 
\end{equation*}%
is a weak solution to problem (\ref{Eq}). As before, there exists $\overline{D}%
\left( \tau \right) >0$ such that 
\begin{equation*}
\Vert \overline{u}(t)\Vert _{\infty }\leq \overline{D}(\tau )\left(
\sup_{s\geq 0}\Vert \overline{u}(s)\Vert _{2}+1\right) =\overline{R}(\tau )%
\text{ if }t\geq \tau .
\end{equation*}%
Take $R_{1}=\max \{R(\tau ),\overline{R}(\tau )\}$. The difference $w\left(
t\right) =u(t)-\overline{u}(t)$ satisfies the equation%
\begin{equation*}
\frac{dw}{dt}-\Delta w+f\left( u\right) -f(\overline{u})=0.
\end{equation*}%
Multiplying it by $w$ and integrating over $\left( 0,T\right) $ we have%
\begin{eqnarray*}
\left\Vert w\left( T\right) \right\Vert _{2}^{2} &\leq
&- \int_{0}^{T}\int_{\Omega }\left( f(u(t,x))-f(\overline{u}%
(t,x))\right) \cdot (u(x,t)-\overline{u}(x,t)) dxdt \\
&\leq &L_{1}\int_{0}^{T}\int_{\Omega }\left\vert w(t,x)\right\vert ^{2}dxdt=L_{1}\int_{0}^{T}\left\Vert w\left( t\right)
\right\Vert _{2}^{2}dt,
\end{eqnarray*}%
where $L_{1}$ is the one-sided Lipschitz constant corresponding to $R_{1}$. Hence, the
Gronwall lemma implies that $w\left( T\right) =0$ for any $T\geq 0$,
proving uniqueness of the solution $v.$
\end{proof}

\bigskip

\begin{theorem}
\label{Unique2}For any initial data $u_{0}\in \mathcal{A}$ the weak solution is unique.
Moreover, for any $u_{0},v_{0}\in \mathcal{A}$ the corresponding weak
solutions satisfy%
\begin{equation}
\left\Vert u\left( t\right) -v\left( t\right) \right\Vert _{2}\leq
\left\Vert u_{0}-v_{0}\right\Vert _{2}e^{Lt}  \label{IneqAttr}
\end{equation}%
for some $L>0.$
\end{theorem}

\begin{proof}
Lemma \ref{BoundedLinf} and Remark \ref{BoundedLinf2} imply that $\left\Vert
y\right\Vert _{\infty }\leq R$ for any $y\in \mathcal{A}$, where $R>0$. Let $%
L$ be the one sided Lipschitz constant corresponding to $R$. Then, arguing as in the
proof of Theorem \ref{Unique} we obtain that any two solutions $u($%
\textperiodcentered $),\ v($\textperiodcentered $)$ starting inside the
attractor satisfy%
\begin{equation*}
\left\Vert u\left( t\right) -v(t)\right\Vert _{2}\leq \left\Vert u\left(
0\right) -v(0)\right\Vert _{2}+L\int_{0}^{t}\left\Vert u\left( s\right)
-v(s)\right\Vert _{2}ds.
\end{equation*}%
The Gronwall lemma gives (\ref{IneqAttr}) and the uniqueness of solutions.
\end{proof}

\bigskip

\subsection{Upper bound on the attractor dimension.}
In order to obtain the estimates on the attractor dimension we need to reinforce the one-sided Lipschitz condition \eqref{Lip}. Namely, we replace it by the stronger standard (two-sided) local Lipschitz condition which states that for every $R>0$ there exists a constant $L(R)>0$ such that
\begin{equation}\label{lip2}
	|f(u)-f(v)|\leq L(R)|u-v|\ \ \textrm{for every}\ \ |u|,|v|\leq R. 
\end{equation}
Note that \eqref{lip2} implies \eqref{Lip} with the same constant $L(R)$. 

Let us recall the definitions of the Hausdorff and fractal dimensions of a
compact subset $\mathcal{A}$ of the Hilbert space $H$. Let $B(a,r)$ be a closed ball of radius $r$ centered at $a$. Let $\mathcal{U}
$ be a covering of $\mathcal{A}$ by a finite family of balls $B(x_{i},r_{i})$ such
that $\sup_{i}\{r_{i}\}=\delta (\mathcal{U})\leq \delta $. Then the $d$%
-dimensional Hausdorff measure of $\mathcal{A}$ is defined by 
\begin{equation*}
\mu _{H}(\mathcal{A},d)=\lim_{\delta \rightarrow 0}\mu _{H}(\mathcal{A}%
,d,\delta ),
\end{equation*}%
where 
\begin{equation*}
\mu _{H}(\mathcal{A},d,\delta )=\inf_{\delta (\mathcal{U})\leq \delta
}\sum_{i}r_{i}^{d}.
\end{equation*}%
It is known that there exists $d=d_{H}(\mathcal{A})\in \lbrack 0,+\infty ]$
such that $\mu _{H}(\mathcal{A},d)=0$ for $d>d_{H}(\mathcal{A})$ and $\mu
_{H}(\mathcal{A},d)=\infty $ for $d<d_{H}(\mathcal{A})$ (see \cite{Falconer}%
). The number $d_{H}(\mathcal{A})$ is called the Hausdorff dimension of $\mathcal{A}$.
The fractal dimension of $\mathcal{A}$ is 
\begin{equation*}
d_{f}(\mathcal{A})=\inf \{d>0\mid \mu _{f}(\mathcal{A},d)=0\},
\end{equation*}%
where 
\begin{equation*}
\mu _{f}(\mathcal{A},d)=\overline{\lim_{\epsilon \rightarrow 0}}\,\epsilon
^{d}n_{\epsilon },
\end{equation*}%
and $n_{\epsilon }$ is the minimum number of balls of radius $r\leq \epsilon 
$ which is necessary to cover $\mathcal{A}$. Since $\mu _{H}(\mathcal{A}%
,d)\leq \mu _{f}(\mathcal{A},d),$ we deduce that $d_{H}(\mathcal{A})\leq
d_{f}(\mathcal{A})$, the converse being false in general.

We recall the following abstract result proved in \cite{BalVal}, which is a
modification of a previous result from \cite{Lad}.

\begin{theorem}
\label{DimJose2}Let $V:H\rightarrow H$ be a continuous mapping such that $%
\mathcal{A=}V(\mathcal{A})$. Let us suppose that there exist $l\in \lbrack
1,+\infty ),\,\delta \in (0,\frac{1}{\sqrt{2}})$ such that for any $u,v\in 
\mathcal{A}$%
\begin{equation}
\left\Vert V(u)-V(v)\right\Vert \leq l\left\Vert u-v\right\Vert ,
\label{DimLip}
\end{equation}%
\begin{equation}
\left\Vert Q_{N}V(u)-Q_{N}V(v)\right\Vert \leq \delta \left\Vert
u-v\right\Vert ,  \label{DimContr}
\end{equation}%
where $Q_{N}$ is the projector in $H$ into some subspace $H_{N}^{\perp }$ of
codimension $N\in \mathbb{N}$. Then for any $\eta >0$ such that $\left( 
\sqrt{2}6l\right) ^{N}\left( \sqrt{2}\delta \right) ^{\eta }=\sigma <1$ the
next inequality holds 
\begin{equation}
d_{H}(\mathcal{A})\leq d_{f}(\mathcal{A})\leq N+\eta .  \label{EstFractal}
\end{equation}%
Hence,%
\begin{equation}
d_{H}(\mathcal{A})\leq d_{f}(\mathcal{A})\leq N\left( 1-\frac{\log (\sqrt{2}%
\delta )}{\log (\sqrt{2}6l)}\right) .  \label{EstFracta2}
\end{equation}
\end{theorem}

Let us apply now this theorem in order to estimate the dimension of the
global attractor for problem (\ref{Eq}).

\begin{theorem}
Let $\mathcal{A}$ be the global attractor of the semiflow $G$. Then%
\begin{equation}
d_{H}(\mathcal{A})\leq d_{f}(\mathcal{A})\leq CL^{\frac{d}{2}},
\label{Fractal}
\end{equation}%
for some $C>0$, where $L$ is the Lipschitz constant from \eqref{lip2} corresponding to $R$ such that $\|y\|_\infty\leq R$ for every $y\in \mathcal{A}$.
\end{theorem}

\begin{proof}
We will check the existence of $t>0$ such that assumptions (\ref{DimLip})-(%
\ref{DimContr}) hold true for the operator $V=G(t,$\textperiodcentered $)$
in the space $H=L^{2}(\Omega )$.

In view of (\ref{IneqAttr}), condition (\ref{DimLip}) follows with $%
l(t)=e^{Lt}$ for any $t>0.$

Let $0<\lambda _{1}\leq \lambda _{2}\leq \lambda _{3}\leq ...$ be the
eigenvalues of the operator $-\Delta $ with Dirichlet boundary conditions
and let $\{\varphi _{k}\}_{k=1}^{\infty }$ be their corresponding
eigenfunctions. We define $H_{N}$, the $N$-dimensional subspace of $H$
generated by $\{\varphi _{1},,,,,\varphi _{N}\}$. Accordingly, we define the
$L^2(\Omega)$ orthogonal projectors $P_{N}:H\rightarrow H_{N}$, $Q_{N}:H\rightarrow H_{N}^{\bot }$.

For two solutions $u($\textperiodcentered $),\ v($\textperiodcentered $)$ we
multiply the equality%
\begin{equation*}
\frac{dw}{dt}-\Delta w+f\left( u\right) -f(v)=0
\end{equation*}%
by $Q_{N}w(t)$, where $w(t)=u(t)-v(t).$ Then%
\begin{eqnarray*}
\frac{1}{2}\frac{d}{dt}\left\Vert Q_{N}w\right\Vert _{2}^{2}+\left\Vert
\nabla Q_{N}w(t)\right\Vert _{2}^{2} &\leq &\int_{\Omega }\left\vert
f(u(t,x))-f(v(t,x))\right\vert \left\vert Q_{N}w(t,x)\right\vert dx \\
&\leq &L\left\Vert w\left( t\right) \right\Vert _{2}\left\Vert
Q_{N}w\right\Vert _{2}.
\end{eqnarray*}%
As $\left\Vert \nabla Q_{N}w(t)\right\Vert _{2}^{2}\geq \lambda
_{N+1}\left\Vert Q_{N}w\right\Vert _{2}^{2}$ and $\left\Vert w\left(
t\right) \right\Vert _{2}\leq e^{Lt}\left\Vert w\left( 0\right) \right\Vert
_{2}$, we get%
\begin{equation*}
\frac{d}{dt}\left\Vert Q_{N}w\right\Vert _{2}+\lambda _{N+1}\left\Vert
Q_{N}w\right\Vert _{2}\leq Le^{Lt}\left\Vert w\left( 0\right) \right\Vert
_{2},
\end{equation*}%
so%
\begin{equation*}
\frac{d}{dt}(e^{\lambda _{N+1}t}\left\Vert Q_{N}w\right\Vert _{2})\leq
Le^{(L+\lambda _{N+1})t}\left\Vert w\left( 0\right) \right\Vert _{2}
\end{equation*}%
and integrating over $\left( 0,t\right) $ we have%
\begin{eqnarray*}
\left\Vert Q_{N}w(t)\right\Vert _{2} &\leq &\left\Vert Q_{N}w(0)\right\Vert
_{2}e^{-\lambda _{N+1}t}+e^{-\lambda _{N+1}t}\int_{0}^{t}Le^{(L+\lambda
_{N+1})s}\left\Vert w\left( 0\right) \right\Vert _{2}ds \\
&=&\left\Vert Q_{N}w(0)\right\Vert _{2}e^{-\lambda _{N+1}t}+\frac{L}{%
L+\lambda _{N+1}}\left\Vert w\left( 0\right) \right\Vert _{2}e^{Lt} \\
&\leq &\left(e^{-\lambda _{N+1}t}+\frac{L}{L+\lambda _{N+1}}e^{Lt}\right)\left\Vert w\left(
0\right) \right\Vert _{2} \\
&=&\delta (t,N)\left\Vert w\left( 0\right) \right\Vert _{2}.
\end{eqnarray*}%
Since $\lambda _{N}\rightarrow +\infty $, we can choose $N$ and $t$ such
that $\delta (t,N)<\frac{1}{\sqrt{2}}$. Then, for $\eta $ satisfying $\left( 
\sqrt{2}6l\right) ^{N}\left( \sqrt{2}\delta \right) ^{\eta }<1$ we obtain
that%
\begin{equation*}
d_{H}(\mathcal{A})\leq d_{f}(\mathcal{A})\leq N+\eta ,
\end{equation*}%
proving that the Hausdorff and fractal dimension are finite.

It remains to prove (\ref{Fractal}). To this end, we choose $\eta =N$ and $t,\ N$ such that $12l(t)\delta (t,N)<1$.
Hence, we seek for the numbers $t$ and $N$ such that%
\begin{equation}
12l(t)\delta (t,N)=12\left( e^{(L-\lambda _{N+1})t}+\frac{Le^{2Lt}}{%
L+\lambda _{N+1}}\right) <1.  \label{Ineqldelta}
\end{equation}%
We can choose $t$ and $N$ such that the following two relations hold at the same time%
\begin{equation*}
12e^{(L-\lambda _{N+1})t}=\frac{1}{2},
\end{equation*}%
\begin{equation*}
\frac{Le^{2Lt}}{L+\lambda _{N+1}}\leq \frac{1}{24+\alpha },
\end{equation*}%
for some $\alpha >0$, so (\ref{Ineqldelta}) holds true. Hence, 
\begin{equation*}
t=\frac{\log (24)}{\lambda _{N+1}-L}
\end{equation*}%
and%
\begin{equation*}
(\lambda _{N+1}-L)\log \left( \frac{L+\lambda _{N+1}}{L(24+\alpha )}\right)
\geq 2 L\log (24).
\end{equation*}%
Since $\lambda _{N}=O(N^{\frac{2}{d}})$ (see \cite[p.201]{Brezis} or \cite[%
p.136]{Metivier}), we find that $\lambda _{N}\geq DN^{\frac{2}{d}}$ for some 
$D>0$. Thus,%
\begin{equation*}
(\lambda _{N+1}-L)\log \left( \frac{L+\lambda _{N+1}}{L(24+\alpha )}\right)
\geq (D(N+1)^{\frac{2}{d}}-L)\log \left( \frac{L+D(N+1)^{\frac{2}{d}}}{%
L(24+\alpha )}\right) .
\end{equation*}%
Therefore, it is enough to choose $N$ such that at the same time we have the following two iequalities%
\begin{equation*}
N\geq \left( \frac{3L}{D}\right) ^{\frac{d}{2}}-1,
\end{equation*}%
\begin{equation*}
\frac{L+D(N+1)^{\frac{2}{d}}}{L(24+\alpha )}\geq 24.
\end{equation*}%
The last inequality is equivalent to%
\begin{equation*}
N\geq \left( \frac{L}{D}(24(24+\alpha )-1)\right) ^{\frac{d}{2}}-1.
\end{equation*}%
Put $K=(24(24+\alpha )-1)/D$. Then we choose $N=[(KL)^{\frac{d}{2}}]$, where 
$[x]$ is the integer part of $x$.

Note that we can assume that $N\geq 1$, because if $N=0$, then it follows
from the previous inequalities that $\lambda _{1}\geq 3L>L$, so it is easy
so see that the global attractor reduces to one point and then $d_{H}(%
\mathcal{A})=d_{f}(\mathcal{A})=0.$ Hence, $\delta (t,N)<1/\sqrt{2},\ \left(
12l(t)\delta (t,N)\right) ^{N}<1$ and Theorem \ref{DimJose2} implies that $%
d_{H}(\mathcal{A})\leq d_{f}(\mathcal{A})\leq 2N$. Finally, we obtain that%
\begin{equation*}
d_{H}(\mathcal{A})\leq d_{f}(\mathcal{A})\leq 2N\leq 2(KL)^{\frac{d}{2}}=CL^{%
\frac{d}{2}},
\end{equation*}%
where $C=2K^{\frac{d}{2}}$.
\end{proof}

\bigskip

\textbf{Acknowledgments}

This research was supported by FEDER and the Spanish Ministerio de Ciencia e
Innovaci\'{o}n under projects PDI2021-122991-NB-C21 and Proyecto
PID2019-108654GB-I00, and by the Generalitat Valenciana, project
PROMETEO/2021/063. Research of PK was moreover supported by
Polish National Agency for Academic Exchange (NAWA) within the Bekker Programme under Project No.
PPN/BEK/2020/1/00265/U/00001, and by National Science Center (NCN) of Poland under Project No. DEC-2017/25/B/ST1/00302.

\end{document}